\documentclass{amsart}
\usepackage{pdfsync}
\usepackage{amsmath}%
\usepackage{amsfonts}%
\usepackage{amssymb}%
\usepackage{graphicx}
\usepackage{subfigure}
\newtheorem{thm}{Theorem}[section]
\theoremstyle{plain}

\newtheorem{lemma}[thm]{Lemma}
\newtheorem{prop}[thm]{Proposition}
\theoremstyle{remark}
\newtheorem{rem}[thm]{Remarks}
\newtheorem{remark}[thm]{Remark}

\newcommand{\Tone}{\mathbb{T}}
\newcommand{\rs}{r_{\star}}
\newcommand{\trho}{\tilde{\rho}}
\newcommand{\brho}{\bar{\rho}}

\newcommand{\bG}{\bar{\mathcal{G}}}
\newcommand{\R}{\mathbb R}

\numberwithin{equation}{section}
\begin{document}
\title[ASD well-posedness, stability, and bifurcation]{On well-posedness, stability, and bifurcation for the axisymmetric surface diffusion flow}

\author[J. LeCrone]{Jeremy LeCrone}
\address{Department of Mathematics \\
 		Vanderbilt University \\
		 Nashville, TN USA}
\email{jeremy.lecrone@vanderbilt.edu}%
\urladdr{http://www.vanderbilt.edu/math/people/lecrone/}

\author[G. Simonett]{Gieri Simonett}
\address{Department of Mathematics \\
		Vanderbilt University\\
		Nashville, TN USA}
\email{gieri.simonett@vanderbilt.edu}%
\urladdr{http://www.vanderbilt.edu/math/people/simonett}


\begin{abstract}
We study the axisymmetric surface
diffusion flow (ASD), a fourth-order geometric evolution law. 
In particular, we prove that ASD generates a real analytic semiflow in the space of 
$(2 + \alpha)$-little-H\"older regular surfaces of revolution embedded in 
$\mathbb{R}^3$ and satisfying periodic boundary conditions.
Further, we investigate the geometric properties of solutions to ASD.
Utilizing a connection to axisymmetric surfaces with constant mean curvature,
we characterize the equilibria of ASD. Then, focusing on the family of 
cylinders, we establish results regarding stability, instability 
and bifurcation behavior, with the radius acting as
a bifurcation parameter.
\end{abstract}
\maketitle

\section{Introduction}

The central focus of this article is the development of an analytic setting 
for the axisymmetric surface diffusion flow (ASD) with periodic boundary conditions.
We establish well-posedness of ASD and investigate geometric properties of solutions, 
including characterizing equilibria and studying their stability, instability and bifurcation behavior.
We establish and take full advantage of \emph{maximal regularity} 
for ASD. Most notably, with maximal regularity we gain access to the implicit function 
theorem, a very powerful tool in nonlinear analysis and dynamical systems theory. 
We begin with a motivation and derivation of the general surface diffusion flow, 
of which ASD is a special case, and we introduce the general outline of the paper.

The mathematical equations modeling surface diffusion go back to a paper by Mullins \cite{MUL57} 
from the 1950s, who was in turn motivated by earlier work of Herring \cite{HER51}. Both of these authors
investigate phenomena witnessed in sintering processes, a method by which
objects are created by heating powdered material to a high temperature, while remaining below the melting
point of the particular substance. When the applied temperature reaches a critical value, the 
atoms on the surfaces of individual particles will diffuse across to other particles, fusing the
powder together into one solid object. In response to gradients of the chemical potential 
along the surface of this newly formed object, the surface atoms may undergo diffusive mass transport 
\emph{on} the surface of the object, attempting to reduce the surface free energy. Given the right 
conditions -- temperature, pressure, grain size, sample size, etc. -- the mass flux due to this chemical
potential will dominate the dynamics and it is the resulting morphological evolution 
of the surface which the surface diffusion flow aims to model.
We also note that the surface diffusion flow has been used to model the motion of surfaces 
in other physical processes (e.g. growth of crystals and nano-structures). 
The article \cite{CT94} contains the formulation of the model which we present  
below, which is set in a more general framework than the original model developed by Mullins. 

\subsection{The Surface Diffusion Flow}\label{Sec:IntroSD}


From a mathematical perspective, the governing equation for motion via surface diffusion can be expressed
for hypersurfaces in arbitrary space dimensions. In particular, let $\Gamma \subset \mathbb{R}^n$ be a 
closed, compact, immersed, oriented Riemmanian manifold with codimension 1. Then we denote by 
$\mathcal{H} = \mathcal{H}(\Gamma)$ the (normalized) mean curvature on $\Gamma$, which is simply the 
sum of the principle curvatures on the hypersurface, and $\Delta_{\Gamma}$ denotes the Laplace--Beltrami
operator, or \emph{surface Laplacian}, on $\Gamma$. 
The motion of the surface $\Gamma$ by surface diffusion is then governed by the equation
\[
V = \Delta_{\Gamma} \mathcal{H},
\]
where $V$ denotes the normal velocity of the surface $\Gamma$. 
If $\Gamma$ encloses a region $\Omega$ we assume the unit normal field to be pointing outward.
%
A solution to the surface diffusion problem on the interval $J \subset \mathbb{R}_+$, with 
$0 \in J$, is a family $\{ \Gamma(t): t \in J \}$ of closed, compact, immersed hypersurfaces in $\mathbb{R}^n$ 
which satisfy the equation
\begin{equation}\label{SD}
\begin{cases}
V(\Gamma(t)) = \Delta_{\Gamma(t)} \mathcal{H}(\Gamma(t)), &\text{$t \in \dot{J} := J \setminus \{ 0 \},$}\\
\Gamma(0) = \Gamma_0,
\end{cases}
\end{equation}
for a given initial hypersurface $\Gamma_0$. It can be shown that solutions to \eqref{SD} 
are volume--preserving, in the sense that the signed volume of the region $\Omega$
is preserved along solutions. Additionally, \eqref{SD} is surface--area--reducing. 
It is also interesting to note that the surface diffusion flow can be viewed as the $H^{-1}$-gradient flow 
of the area functional,
a fact that was first observed in \cite{Fi00}. This particular structure has been exploited in 
\cite{May00,May01} for devising numerical simulations.



For well-posedness of \eqref{SD} we mention \cite{EMS98},  
where it is shown that 
\eqref{SD} admits a unique 
local solution for any initial surface $\Gamma_0\in C^{2+\alpha}$.
Additionally, the authors of \cite{EMS98} show that any initial surface 
that is a small $C^{2+\alpha}$--perturbation of a sphere admits a global
solution which converges to a sphere at an exponential rate.
This result was improved in \cite{EMu10} to admit initial surfaces
in the Besov space $B^s_{p,2}$ with $s=5/2-4/p$ and $p>(2n+10)/3$.
For dimensions $n<7$ note that this allows for initial surfaces 
which are less regular than $C^2$. An independent theory for
existence of solutions to higher--order geometric evolution equations,
which also applies to the surface diffusion flow,
was developed in \cite{HuiPo99, Po96};
see \cite{McCWW11,Sha04} for a discussion of some limitations of these results.

More recent results for initial surfaces with low regularity are contained in \cite{As12, KoLa12}.
The author of \cite{As12} obtains 
existence and uniqueness of local solutions for various geometric evolution 
laws (including the surface diffusion flow),
in the setting of entire graphs with initial regularity $C^{1+\alpha}$. 
Surface diffusion is also one of several evolution laws for
which the authors of \cite{KoLa12} establish solutions 
under very weak, and possibly optimal, regularity assumptions on initial data.
In particular, their results guarantee existence and uniqueness of global 
analytic solutions, in the setting of entire graphs over $\R^n$,  
for Lipschitz initial data $u_0$ with small Lipschitz constant $\|\nabla u_0\|_{L_\infty(\R^n)}$.

For interesting new developments regarding lower bounds on the existence time
of solutions to the surface diffusion flow in $\R^3$ and $\R^4$, we refer the reader to
\cite{McCWW11, Wh11, Wh12}.
In particular, it is shown in \cite{Wh12} that the 
flow of a surface in $\mathbb{R}^3$ which
is initially close to a sphere in $L_2$ (that is, the  $L_2$-norm of the trace-free 
part of the second fundamental form is sufficiently small) 
is a family of embeddings that exists globally and converges at
an exponential rate to a sphere. 
The results of \cite{McCWW11, Wh11} regarding concentration of curvature
along solutions may prove important in the analytic investigation of solutions approaching
finite--time pinch--off.



In the context of geometric evolution equations, 
such as the mean curvature flow \cite{Hu84, Hu87}, 
the surface diffusion flow, or the Willmore flow \cite{KS01, KS02},
the underlying governing equations are often expressed 
by evolving a smooth family of immersions 
$X:M\times (0,T)\to\R^n,$
where $M$ is a fixed smooth oriented manifold and $\Gamma(t)$
is the image of $M$ under $X(\cdot,t)$.
In this formulation, the surface diffusion flow is given by
\begin{equation}
\label{immersion}
\partial_t X=(\Delta H)\nu,\qquad X(\cdot,0)M=\Gamma_0,
\end{equation}
where $\nu$ is the normal to the surface $\Gamma(t)$.
This formulation is invariant under the group of sufficiently smooth diffeomorphisms of $M$,
and this implies that \eqref{immersion} is only weakly parabolic.
A way to infer that \eqref{immersion} is not parabolic is to observe 
that if $X(\cdot,t)$ is a solution, then 
so is $X(\phi(\cdot),t)$, for any diffeomorphism $\phi$.
Given a smooth solution $X$ one can therefore
construct nonsmooth (i.e.\ non $C^\infty$) solutions by choosing $\phi$
to be nonsmooth. If \eqref{immersion} were parabolic, all solutions would
have to be smooth, as was pointed out in~\cite{An94} for the mean curvature flow.
%

Nevertheless, existence of unique smooth solutions for the mean curvature flow, 
for compact $C^\infty$--initial surfaces $X(\cdot,0)M$, 
can be derived by making use of the Nash-Moser implicit function theorem, 
see for instance \cite{GaHa86, Ha82a}.

The Nash-Moser implicit function theorem 
may also lead to a successful treatment of the 
surface diffusion flow \eqref{immersion}.
However, there is an alternative approach 
to dealing with the motion of surfaces by curvature
(for example the mean and volume--preserving mean curvature flows, 
the surface diffusion flow, the Willmore flow) which removes the issue
of randomness of a parameterization:
if one fixes the parameterization as a graph in normal direction with respect to a 
reference manifold and then expresses the governing equations in terms of
the graph function, the resulting equations are
{\bf quasilinear} and {\bf strictly parabolic}.
This approach has been employed in \cite{EMS98,ES98b,ES98c,Sim01}, 
and also in \cite{HuiPo99}. 
In the particular case of the surface diffusion flow, 
one obtains a fourth--order quasilinear parabolic evolution equation.
One can then apply well--established results for quasilinear parabolic equations. 
The theory in \cite{Am93,CS01}, for instance, works for any quasilinear parabolic evolution equation,
no matter whether it is cast as a more traditional PDE in Euclidean space, 
or an evolution equation living on a manifold. 
This theory also renders access to well--known principles from dynamical systems.

The approach of parameterizing the unknown 
surface as a graph in normal direction has also been applied to a wide array of 
free boundary problems, including problems in phase transitions
(where the graph parameterization and its extension into the bulk phases 
is often referred to as the Hanzawa transformation), 
see for example \cite{PSZ12} and the references therein.


The literature on geometric evolution laws often
considers the question of short-time existence standard
and refers to the classical monographs \cite{Ei69, EZ98, LSU}.
However, when the setting is a manifold rather than Euclidean space,
existence theory for parabolic (higher order) equations does not belong to the standard theory
and requires a proof, a point that is also acknowledged in \cite{HuiPo99}, see page 61.



\subsection{Axisymmetric Surface Diffusion (ASD)}\label{Sec:IntroASD}


For the remainder of the paper, we focus our attention on 
the case of $\Gamma \subset \mathbb{R}^3$ an embedded
surface which is symmetric about an axis of rotation (which we take to be the $x$--axis,
without loss of generality) and satisfies prescribed periodic boundary conditions on some fixed
interval $L$ of periodicity (we take $L = [-\pi, \pi]$ and enforce $2 \pi$ periodicity,
without loss of generality). In particular, the axisymmetric surface $\Gamma$  
is characterized by the parametrization
\[
\Gamma = \Big\{ (x, r(x) \cos(\theta), r(x) \sin(\theta)): \; x \in \mathbb{R}, \; \theta \in [-\pi, \pi] \Big\},
\]
where the function $r : \mathbb{R} \rightarrow (0, \infty)$ is the \emph{profile function}
for the surface $\Gamma$. Conversely, a profile function $r: \mathbb{R} \rightarrow (0, \infty)$
generates an axisymmetric surface $\Gamma = \Gamma(r)$ via the parametrization given above. 

We thus recast the surface
diffusion problem as an evolution equation for the profile functions $r = r(t)$. In particular,
one can see that the surface $\Gamma(r)$ inherits the Riemannian metric
\[
g = (1 + r_x^2) \, dx \wedge dx + r^2 \, d \theta \wedge d \theta, 
\]
from the embedding $\Gamma \hookrightarrow \mathbb{R}^3$, with respect to the surface coordinates 
$(x, \theta)$; where the subscript $f_{x_i} := \partial_{x_i} f$ denotes the derivative of
$f$ with respect to the indicated variable $x_i$. It follows that the (normalized) 
mean curvature $\mathcal{H}(r)$ of the surface is given by $\mathcal{H}(r) = \kappa_1 + \kappa_2$, where
\[
\kappa_1 = \frac{1}{r \sqrt{1 + r_x^2}} \quad \text{and} \quad \kappa_2 = \frac{- r_{xx}}{(1 + r_x^2)^{3/2}}
\]
are the \emph{azimuthal} and \emph{axial} principle curvatures, respectively, on $\Gamma(r)$. 
Meanwhile, the Laplace--Beltrami operator on $\Gamma$ and the normal velocity of $\Gamma = \Gamma(t)$ are 
\begin{align*}
\Delta_{\Gamma(r)} &= \frac{1}{r \sqrt{1 + r_x^2}} \left( \partial_x \left[ \frac{r}{\sqrt{1 + r_x^2}} \partial_x \right] + 
\partial_{\theta} \left[ \frac{\sqrt{1 + r_x^2}}{r} \partial_{\theta} \right] \right),\\
V(t) &= \frac{r_t}{\sqrt{1 + r_x^2}} \, .
\end{align*}
Finally, substituting these 
terms into the equation \eqref{SD} and simplifying, we arrive at the expression 
\begin{equation}\label{ASD}
\left\{
\begin{aligned}
&r_t = \frac{1}{r} \, \partial_x \left[ \frac{r}{\sqrt{1 + r_x^2}} \; 
\partial_x \left(\frac{1}{r \sqrt{1 + r_x^2}} - 
\frac{r_{xx}}{(1 + r_x^2)^{\frac{3}{2}}} \right) \right], \quad &\text{$t > 0, \, x \in \mathbb{R},$}\\
&r(t,x + 2 \pi) = r(t,x), & \text{$t \geq 0, \, x \in \mathbb{R}$},\\
&r(0,x) = r_0(x),  & \text{$ x \in \mathbb{R},$}
\end{aligned} 
\right.
\end{equation}
for the periodic axisymmetric surface diffusion problem.
To simplify notation in the sequel, we define the operator
\begin{align}\label{Eqn:GDefined}
G(r) &:= \frac{1}{r} \; \partial_x \left[ \frac{r}{\sqrt{1 + r_x^2}} \; \partial_x  
\mathcal{H}(r) \right],
\end{align} 
which is formally equivalent to the right hand side of the first equation in \eqref{ASD}.

\medskip
\goodbreak

The main results of this paper address
\begin{itemize}
\item[(a)] existence, uniqueness, and regularity of solutions for \eqref{ASD},
\item[(b)] nonlinear stability and instability of equilibria for \eqref{ASD}, 
\item[(c)] bifurcation of equilibria from the family of cylinders, with the radius serving as bifurcation parameter.
\end{itemize}
As mentioned in the introduction, we develop and take full advantage of maximal regularity for ASD.
In this setting, the results in (a) follow in a straight forward way from \cite{CS01}.
Existence results could also be based on the approach developed in \cite{As12, KoLa12}, but
we prefer to work within the well-established framework of continuous maximal regularity. 
It provides a general and flexible setting for investigating further qualitative properties of solutions. 
The novelty of the results in (b) is analysis of the {\em nonlinear} structure
of solutions. 
Corresponding results for {\em linear} stability and instability
of equilibria are contained in \cite{BBW98} where a precise characterization of the eigenvalues
of the linearized problem is given.
Based on a formal center manifold analysis, the authors in \cite{BBW98} predict subcritical 
bifurcation of equilibria at the critical value of radius $\rs = 1$, 
but no analytical proof is provided.
Thus, our result in (c) appears to be the first rigorous proof of bifurcation. 
In addition, we show that the bifurcating equilibria (which are shown to coincide with the Delaunay unduloids) 
are {\em nonlinearly} unstable. We note that previous results show {\em linear} instability of unduloids
and we refer the reader to Remark~\ref{remark-bifurcation} for a more detailed discussion.

The publication \cite{BBW98} has served as a source of inspiration for our investigations.
It provides an excellent overview of the complex qualitative behavior of ASD, 
with results supported by analytic arguments and numerical computations.


The first investigations of evolution of an axisymmetric surface via surface diffusion 
can be traced back to the work of Mullins and Nichols \cite{MN65, MN65a} in 1965, where one can 
already see some of the benefits of this special setting. 
Taking advantage of the symmetry of the problem, they  
developed an adequate scheme for numerical methods 
and they already predicted the finite time pinch--off of tube--like 
surfaces via surface diffusion, a feature similar to the mean curvature flow and a natural
phenomenon to study in exactly this axisymmetric setting.  
Research continued to focus on pinch--off behavior using numerical methods, 
c.f. \cite{BBW98, CFM95, CFM95a, CFM96, DMVW, MS79, MS82}, 
wherein many schemes were developed to handle the continuation of solutions after the change of 
topology at the moment of pinch--off. Unlike the related behavior for the mean curvature flow, 
pinch--off for the surface diffusion flow remains a numerical observation that has yet
to be verified analytically.

Much research has also focused on the numerical investigation of 
stability and instability of cylinders
under perturbations of various types, see \cite{BBW98, CFM95, CFM96} for instance.
In an important construction from \cite{CFM96}, the authors observe destabilization of a
particular perturbation of a cylinder (i.e. divergence from the cylinder) due to second--order effects
of the flow, whereas the first--order (linear) theory predicts asymptotic stability.
In fact, their formulation produces conditions under which a perturbation will destabilize
due to $n^{\text{th}}$--order effects, where $(n-1)^{\text{st}}$--order analysis predicts stability.
This result highlights the importance of studying the full nonlinear behavior of solutions to ASD.

We proceed with an outline of the article and description of our main results.
In Section~\ref{Sec:WellPosedness}, we establish existence of solutions to \eqref{ASD}
in the framework of continuous maximal regularity. 
In particular, we have existence and uniqueness of maximal solutions
for initial surfaces which are $(2 + \alpha)$--little--H\"older continuous.
Solutions are also analytic in time and space,
for positive time, with a prescribed singularity at time $t = 0$. 
Additionally, we state conditions for global existence 
of the semiflow induced by \eqref{ASD}. 
We rely on the theory developed in \cite{LeC11} and the 
well--posedness results for quasilinear equations with maximal regularity provided in \cite{CS01}.
We include comments on how we prove these well--posedness results in an appendix.

In Section~\ref{Sec:Equilibria}, we characterize 
the equilibria of ASD using results of Delaunay \cite{DE41} and Kenmotsu \cite{KEN80} 
regarding constant mean curvature surfaces in the axisymmetric setting. 
We conclude that all equilibria of \eqref{ASD} must fall into the family of undulary curves, 
which includes all constant functions $r(x) \equiv \rs > 0$ (corresponding to the cylinder of radius $\rs$)
and the two--parameter family of nontrivial undulary curves $R(B,k)$. 

In Section~\ref{Sec:Stability}, we prove that the family of cylinders
with radius $\rs > 1$ are asymptotically, exponentially stable under a large class of nonlinear perturbations,
which maintain the same axis of symmetry and satisfy the prescribed periodic boundary conditions.
In particular, given $\rs > 1$, we prove that any sufficiently small $(2 + \alpha)$--little--H\"older regular
perturbation produces a global solution which converges exponentially fast to the 
cylinder of radius $\rs + \eta > 1$. The value $\eta$ is determined by the volume
enclosed by the perturbation, which may differ from the volume of the original cylinder.
In proving this result, we note that the spectrum of the linearized 
equation at $\rs$ is contained in the \emph{left half} of the complex plane,
though it will always contain 0 as an eigenvalue.
By reducing the equation, essentially to the setting of volume--preserving
perturbations of a cylinder, we are able to eliminate the zero eigenvalue.
We then prove nonlinear stability in the reduced setting, 
utilizing maximal regularity methods on exponentially weighted function spaces, 
and we transfer the result back to the (full) problem via a \emph{lifting} operator. 

In Section~\ref{Sec:Instability}, we prove nonlinear instability 
of cylinders with radius $0 < \rs < 1$.
We take nonlinear instability to be the logical negation of stability,
which one may interpret as the existence of at least one {\em unstable perturbation},
see Theorem~\ref{Thm:Instability} for a precise statement.
This result makes use of a contradiction 
technique reminiscent of results from the theory of ordinary differential equations, 
c.f. \cite{PW10}. By isolating the linearization of the governing equation, 
one takes advantage of a \emph{spectral gap} and associated spectral 
projections in order to derive necessary conditions for {\em stable perturbations}, 
which in turn lead to a contradiction. 

We note that previous instability results for ASD have focused
primarily on classifying stable and unstable {\em eigenmodes} of equilibria,
which gives precise results on the behavior of solutions associated with unstable
perturbations. However, these methods are limited to 
the behavior of solutions under the linearized flow.

Finally, in Section~\ref{Sec:Bifurcation} we apply classic methods of Crandall and Rabinowitz \cite{CR71} 
to verify the subcritical bifurcation structure of all points of intersection between the family
of cylinders and the disjoint branches of unduloids. In particular, taking the inverse of the radius
$\lambda = 1 / \rs$ as a bifurcation parameter, we verify the existence of continuous families 
of nontrivial equilibria which branch off of the family of cylinders at radii 
$\rs = 1 / \ell$, for all $\ell \in \mathbb{N}.$  
We conclude that each of these branches corresponds to the branch $R(B,\ell)$ of 
$2 \pi / \ell$--periodic undulary curves. 
Working in the reduced setting established in Section~\ref{Sec:Stability}, it turns out that 
eigenvalues associated with the linearized problem are not simple, 
hence we cannot directly apply the results of \cite{CR71}.
However, restricting attention to surfaces which are even (symmetric about the surface $[x = 0]$), 
we eliminate redundant eigenvalues,
similar to a method used by Escher and Matioc \cite{EM11}.
In this even function setting, we have simple eigenvalues and derive bifurcation
results, which we apply back to the full problem via a posteriori symmetries of equilibria. 

Using eigenvalue perturbation methods, we are also able to conclude nonlinear instability
of nontrivial unduloids, using the same techniques as in Section~\ref{Sec:Instability}.
We once more refer to Remark~\ref{remark-bifurcation} for more information.

In future work we plan to investigate well-posedness of ASD under weaker
regularity assumptions on the initial data. This will allow for a better understanding 
of global existence, and obstructions thereof.
In particular, we conjecture that solutions developing singularities will have to go through
a pinch-off. 

We also plan to consider non--axisymmetric surfaces. In particular, we plan to investigate 
the stability of cylinders under non-axisymmetric perturbations.

Other interesting questions involve the existence and nature of unstable families of perturbations
and reformulations of the problem to allow for different boundary conditions and immersed surfaces.
In particular, reformulating ASD in terms of parametrically defined curves in $\mathbb{R}^2$
would allow for consideration of immersed surfaces of revolution, 
a setting within which the branches of $2 \pi/\ell$--periodic nodary curves
would be added to the collection of equilibria. 
See Section~\ref{Sec:Equilibria} for a definition and graphs of nodary curves. 

\medskip

Throughout the paper we will use the following notation:
If $E$ and $F$ are arbitrary Banach spaces,
$\mathbb{B}_E(a,r)$ denotes the open ball in $E$ with center $a$ and radius $r>0$
and $\mathcal{L}(E,F)$ consists of all bounded linear
operators from $E$ into $F$. 
For $U\subset E$ an open set, we denote by $C^\omega(U,F)$ the
space of all real analytic mappings from $U$ into $F$.

\subsection{Maximal Regularity}\label{MaxReg}

We briefly introduce (continuous) maximal regularity, 
also called \emph{optimal regularity} in the literature. 
Maximal regularity has received a lot of attention in connection 
with parabolic partial differential equations and evolution laws, 
c.f. \cite{AM95, AM05, AN90, CS01, KPW10, LUN95, Pru03, Sim95}. 
Although maximal regularity can be developed in a 
more general setting, we will focus on the setting of \emph{continuous} maximal regularity and 
direct the interested reader to the references \cite{AM95, LUN95} for a general development of the theory.

Let $\mu \in (0,1], \, J := [0,T]$, for some $T > 0$, and let $E$ be a (real or complex) Banach space. 
Following the notation of \cite{CS01}, we define spaces of continuous functions on $\dot{J} := J \setminus \{ 0 \}$ 
with prescribed singularity at 0. Namely, define
\begin{equation}\label{Eqn:SingularContinuity}
\begin{aligned}
&BU\!C_{1 - \mu}(J,E) := \bigg\{ u \in C(\dot{J}, E): [t \mapsto t^{1 - \mu} u(t)] \in BU\!C(\dot{J},E) \; \text{and} \\ 
& \hspace{9em} \lim_{t \rightarrow 0^+} t^{1 - \mu} \| u(t) \|_E = 0 \bigg\}, \quad \mu \in (0,1)\\
&\| u \|_{B_{1 - \mu}} := \sup_{t \in J} t^{1 - \mu} \| u(t) \|_E,
\end{aligned}
\end{equation}
where $BU\!C$ denotes the space consisting of bounded, uniformly continuous functions. It is easy 
to verify that $BU\!C_{1 - \mu}(J,E)$ is a Banach space when equipped with the norm 
$\| \cdot \|_{B_{1 - \mu}}$. Moreover, we define the subspace
\[
BU\!C_{1 - \mu}^1(J,E) := \left\{ u \in C^1(\dot{J},E) : u, \dot{u} \in BU\!C_{1 - \mu}(J, E) \right\}, \quad \mu \in (0,1)
\]
and we set
\[
BU\!C_0 (J,E) := BU\!C(J,E) \qquad BU\!C^1_0(J,E) := BU\!C^1(J,E).
\]

Now, if $E_1$ and $E_0$ are a pair of Banach spaces such that $E_1$ is continuously embedded in $E_0$,
denoted $E_1 \hookrightarrow E_0$, we set
\[
\begin{split}
&\mathbb{E}_0(J) := BU\!C_{1 - \mu}(J,E_0), \qquad \mu \in (0,1],\\
&\mathbb{E}_1(J) := BU\!C^1_{1 - \mu}(J,E_0) \cap BU\!C_{1 - \mu}(J,E_1),
\end{split}
\] 
where $\mathbb{E}_1(J)$ is a Banach space with the norm 
\[
\| u \|_{\mathbb{E}_1(J)} := \sup_{t \in \dot{J}} t^{1 - \mu} \Big( \| \dot{u}(t) \|_{E_0} + \| u(t) \|_{E_1} \Big).
\]
It follows that the trace operator $\gamma: \mathbb{E}_1(J) \rightarrow E_0$, defined by 
$\gamma v := v(0)$, is well-defined and we denote by $\gamma \mathbb{E}_1$ the image of 
$\gamma$ in $E_0$, which is itself a Banach space when equipped with the norm
\[
\| x \|_{\gamma \mathbb{E}_1} := \inf \Big\{ \| v \|_{\mathbb{E}_1(J)}: v \in \mathbb{E}_1(J) \, \text{and} \, \gamma v = x \Big\}.
\]

For a bounded linear operator $B \in \mathcal{L}(E_1, E_0)$ which is 
closed as an operator on $E_0$, 
we say $\big( \mathbb{E}_0(J), \mathbb{E}_1(J) \big)$ 
is a \emph{pair of maximal regularity} for $B$ and write $B \in \mathcal{MR}_{\mu}(E_1, E_0)$, if 
\[
\left( \frac{d}{dt} + B, \, \gamma \right) \in \mathcal{L}_{isom}(\mathbb{E}_1(J), \mathbb{E}_0(J) \times \gamma \mathbb{E}_1),
\]
where $\mathcal{L}_{isom}$ denotes the space of bounded linear isomorphisms. In particular,  
$\big( \mathbb{E}_0(J), \mathbb{E}_1(J) \big)$ is a pair of maximal regularity for $B$ 
if and only if for every $(f, u_0) \in \mathbb{E}_0(J) \times \gamma \mathbb{E}_1$, there 
exists a unique solution $u \in \mathbb{E}_1(J)$ to the inhomogeneous Cauchy problem 
\[
\begin{cases}
\dot{u}(t) + Bu(t) = f(t), &t \in \dot{J},\\
u(0) = u_0.
\end{cases}
\]
Moreover, in the current setting, it follows that $\gamma \mathbb{E}_1 \, \dot{=} \, (E_0, E_1)_{\mu, \infty}^0$,
i.e. the trace space $\gamma \mathbb{E}_1$ is topologically equivalent to the noted continuous 
interpolation spaces of Da Prato and Grisvard, c.f. \cite{AM95, CS01, DPG79, LUN95}.

\section{Well-Posedness of \eqref{ASD}}\label{Sec:WellPosedness}
When considering the surface diffusion problem, 
the underlying Banach spaces $E_0$ and $E_1$ in the formulation of
maximal regularity 
will be spacial regularity classes which describe the properties 
of the profile functions $r(t)$. We proceed by defining these regularity classes.
We define the one-dimensional torus $\Tone := [-\pi, \pi]$, where the points $-\pi$ and
$\pi$ are identified, which is equipped with the topology generated by the metric
\[
d_{\Tone}(x,y) := \min \{|x - y|, 2 \pi - |x - y| \}, \qquad x, y \in \Tone.
\]  
There is a natural equivalence between functions defined on $\Tone$ and 2$\pi$-periodic
functions on $\mathbb{R}$ which preserves properties of (H\"older) continuity and differentiability. 
In particular, we will be working with the so-called periodic little-H\"older spaces $h^{\sigma}(\Tone),$
for $\sigma \in \mathbb{R}_+ \setminus \mathbb{Z}$. Definitions and basic properties of periodic little-H\"older
spaces, as well as details on the connection between spaces of functions on $\Tone$ and $2 \pi$-periodic
functions on $\mathbb{R}$ can be found in \cite{LeC11} and the references therein. For the readers convenience,
we provide a brief definition of $h^{\sigma}(\Tone)$ below.

For $k \in \mathbb{N}_0 := \mathbb{N} \cup \{0\}$, denote by $C^{k}(\Tone)$ the Banach space of 
$k$-times continuously differentiable functions $f: \Tone \rightarrow \mathbb{R}$, equipped with the norm 
\[
\| f \|_{C^k(\Tone)} := \sum_{j = 0}^k \| f^{(j)} \|_{C(\Tone)} := \sum_{j = 0}^k \left( \sup_{x \in \Tone} |f^{(j)}(x)| \right).
\]
Moreover, for $\alpha \in (0,1)$ and $k \in \mathbb{N}_0$, we define the space $C^{k + \alpha}(\Tone)$ 
to be those functions $f \in C^k(\Tone)$ such that the $\alpha$-H\"older seminorm 
\[
\big[ f^{(k)} \big]_{\alpha, \Tone} := \sup_{\substack{x, y \in \Tone\\ x \not= y}} \frac{|f^{(k)}(x) - f^{(k)}(y)|}{d_{\Tone}^{\alpha}(x,y)}
\]
is finite. It follows that $C^{k + \alpha}(\Tone)$ is a Banach space when equipped with the norm
\[
\| f \|_{C^{k + \alpha}(\Tone)} := \| f \|_{C^k(\Tone)} + [ f^{(k)} ]_{\alpha, \Tone}.
\]
Finally, we define the periodic little-H\"older space
\[
h^{k + \alpha}(\Tone) := \left\{ f \in C^{k + \alpha}(\Tone) : \lim_{\delta \rightarrow 0} \sup_{\substack{x, y \in \Tone\\ 0 < d_{\Tone}(x,y) < \delta}} \frac{|f^{(k)}(x) - f^{(k)}(y)|}{d_{\Tone}^{\alpha}(x,y)} = 0 \right\},
\]
for $k \in \mathbb{N}_0$ and $\alpha \in (0,1)$ which is a Banach algebra with 
pointwise multiplication of functions and equipped with the norm 
$\| \cdot \|_{h^{k + \alpha}} := \| \cdot \|_{C^{k + \alpha}(\Tone)}$ inherited from
$C^{k + \alpha}(\Tone)$. For equivalent definitions and more properties
of the periodic little-H\"older spaces, see \cite[Section 1]{LeC11}.

In order to make explicit the quasilinear structure of \eqref{ASD}, we reformulate the problem.
By expanding the governing equation we arrive at the formally equivalent problem
\begin{equation}\label{ASD2}
\begin{cases}
\partial_t r(t,x) + \big[\mathcal{A}(r(t))r(t)\big](x) = f(r(t,x)), &\text{$t > 0, \; x \in \Tone$},\\
r(0,x) = r_0(x), &\text{$x \in \Tone$},
\end{cases}
\end{equation}
where, for appropriately chosen functions $\rho$, 
\begin{equation}\label{Eqn:ADefined}
\mathcal{A}(\rho) := \frac{1}{(1 + \rho_x^2)^2} \; \partial_x^4 + 
\frac{2 \rho_x \big( 1 + \rho_x^2 - 5\rho \rho_{xx} \big)}{\rho \big( 1 + \rho_x^2 \big)^3} \; \partial_x^3
\end{equation}
is a fourth-order differential operator with variable coefficients over $\Tone$ and
\begin{equation}\label{Eqn:FDefined}
f(\rho) := \frac{\rho_x^2 - 1}{\rho^2 (1 + \rho_x^2)^2} \; \rho_{xx} + \frac{6 \rho_x^2 - 1}{\rho (1 + \rho_x^2)^3} \; \rho_{xx}^2 + 
\frac{3 - 15 \rho_x^2}{(1 + \rho_x^2)^4} \; \rho_{xx}^3 + \frac{\rho_x^2}{\rho^3 (1 + \rho_x^2)} 
\end{equation}
is a $\mathbb{R}$-valued function over $\Tone$.
Looking at these formal expressions, one can deduce several properties that the functions $\rho$ must satisfy in order
to get good mapping properties for $f$ and $\mathcal{A}$. In particular, we want to choose $\rho$ such that $\rho(x) \not= 0$ for 
all $x \in \Tone$, also we want that the spacial derivatives $\rho_x$ and $\rho_{xx}$ make sense and the products
$\rho^2$, $\rho^3$, $\rho \rho_x^2$, etc. have desired regularity properties. With these conditions in mind, we
proceed with our well-posedness result.

\subsection{Existence and Uniqueness of Solutions}\label{Section:Solutions}

We collect statements of well--posedness results and refer the reader to the appendix
for comments on their proof.
Fix $\alpha \in (0,1)$ and define the spaces of $\mathbb{R}$-valued little--H\"older continuous functions
\begin{equation}\label{Eqn:LittleHolderE}
E_0 := h^{\alpha}(\Tone), \quad E_1 := h^{4 + \alpha}(\Tone), \quad \text{and} \quad
E_{\mu} := (E_0, E_1)_{\mu, \infty}^0,
\end{equation}
where $( \cdot, \cdot)_{\mu, \infty}^0$, for $\mu \in (0,1)$, denotes the continuous interpolation
functor of Da Prato and Grisvard, c.f. \cite{DPG79} or \cite{AM95}. It is well-known that 
the little-H\"older spaces are stable under this interpolation method, in particular we know that
\[
E_{\mu} = h^{4 \mu + \alpha}(\Tone) \quad \text{(up to equivalent norms),} \qquad \text{for} \quad 4 \mu + \alpha \notin \mathbb{Z},
\]
c.f. \cite{LeC11, LUN95}. 
Further, let $V$ be the set of functions $r: \Tone \rightarrow \mathbb{R}$ 
such that $r(x) > 0$ for all $x \in \Tone$ and define $V_{\mu} := V \cap E_{\mu}$ for $\mu \in [0,1]$.
We note that $V_{\mu}$ is an open subset of $E_{\mu}$ for all $\mu \in [0,1]$.

Before we can properly state a result on maximal solutions, we need to introduce one more space of
functions from an interval $J \subset \mathbb{R}_+$ to a Banach space $E$, with prescribed singularity
at zero. Namely, if $J = [0, a)$ 
for $a > 0$, i.e. $J$ is a right-open interval containing 0, then we set
\begin{align*}
C_{1 - \mu}(J,E) &:= \{ u \in C(\dot{J},E): u \in BU\!C_{1 - \mu}([0,T],E), \quad T < \sup J \},\\
C^1_{1 - \mu}(J,E) &:= \{ u \in C^1(\dot{J},E): u, \dot{u} \in C_{1 - \mu}(J,E) \}, \qquad \mu \in (0,1],
\end{align*}
which we equip with the natural Fr\'echet topologies induced by 
$BU\!C_{1 - \mu}([0,T],E)$ and $BU\!C_{1 - \mu}^1([0,T],E)$, respectively.

We list some important properties of the mappings ${\mathcal A}$ and $f$, introduced in \eqref{Eqn:ADefined} and \eqref{Eqn:FDefined}.
\begin{lemma}\label{Lem:AandFReg} Let $\mu \in [1/2,1]$. Then 
\[
(\mathcal{A},f) \in C^{\omega} \bigg(V_{\mu},\; \mathcal{MR}_{\nu}(E_1,E_0) \times E_0 \bigg), \qquad \text{for} \quad \nu \in (0, 1],
\]
where $C^{\omega}$ denotes the space of real analytic mappings between Banach spaces.
\end{lemma}

\begin{prop}[Existence and Uniqueness]
\label{Prop:Existence}
Fix $\alpha \in (0,1)$ and take $\mu \in [1/2, 1]$ so that $4 \mu + \alpha \notin \mathbb{Z}$. 
For each initial value $r_0 \in V_{\mu} := h^{4 \mu + \alpha}(\Tone) \cap [r > 0]$, 
there exists a unique maximal solution 
\[
r(\cdot, r_0) \in C^1_{1 - \mu}(J(r_0), h^{\alpha}(\Tone)) \cap C_{1 - \mu}(J(r_0), h^{4 + \alpha}(\Tone)),
\]
where $J(r_0) = [0, t^+(r_0)) \subseteq \mathbb{R}_+$ denotes the maximal interval of 
existence for initial data $r_0$. Further, it follows that 
\[
\mathcal{D} := \bigcup_{r_0 \in V_{\mu}} J(r_0) \times \{ r_0 \}
\]
is open in 
$\mathbb{R}_+ \times V_{\mu}$ and $\varphi:[(t,r_0) \mapsto r(t,r_0)]$ is an analytic
semiflow on $V_{\mu}$, i.e. using the notation $\varphi^t(r_0) := \varphi(t,r_0)$, the mapping 
$\varphi$ satisfies the conditions
\begin{equation}\label{Eqn:Semiflow}
\begin{split}
&\bullet \; \varphi \in C \big( \mathcal{D}, V_{\mu} \big)\\ 
&\bullet \; \varphi^0 = id_{V_{\mu}}\\
&\bullet \; \varphi^{s + t}(r_0) = \varphi^t \circ \varphi^s(r_0) \quad \text{for $0 \leq s < t^+(r_0)$ and $0 \leq t < t^+(\varphi^s(r_0))$}\\
&\bullet \; \varphi(t, \cdot) \in C^{\omega}(\mathcal{D}_t, V_{\mu}) \quad \text{for $t \in \mathbb{R}_+$ with $\mathcal{D}_t := \{ r \in V_{\mu}: (t,r) \in \mathcal{D} \} \not= \emptyset$.}
\end{split}
\end{equation}
\end{prop}

The results in \cite{CS01} also give the following conditions for global solutions. 

\begin{prop}[Global Solutions]
Let $r_0 \in V_{\mu} := h^{4 \mu + \alpha}(\Tone) \cap [r > 0]$ 
for $\mu \in (1/2, 1]$, such that $4 \mu + \alpha \notin \mathbb{Z}$, 
and suppose there exists
$0 < M < \infty$ so that, for all $t \in J(r_0)$
\begin{align*}
&\bullet \; r(t, r_0)(x) \ge 1/M, \; \forall x \in \Tone, \qquad \text{and}\\
&\bullet \; \| r(t, r_0) \|_{h^{4 \mu + \alpha}(\Tone)} \le M,
\end{align*}
then it must hold that $t^+(r_0) = \infty$, so that $r(\cdot, r_0)$ is a global solution.
Conversely, if $r_0 \in V_{\mu}$ and $t^+(r_0) < \infty$, i.e. the solution breaks down in finite--time,
then one, or both, of the conditions stated must fail to hold.
\end{prop}

We can also state the following result regarding analyticity of the maximal
solutions $r(\cdot, r_0)$ in both space and time. 

\begin{prop}[Regularity of Solutions]
\label{Prop:SolutionRegularity}
Under the same assumptions as in Proposition~\ref{Prop:Existence},
it follows that
\begin{equation}
r(\cdot, r_0) \in C^\omega((0, t^{+}(r_0)) \times \Tone) \qquad \text{for all} \quad r_0 \in V_{\mu}, \quad \mu \in [1/2, 1].
\end{equation}
\end{prop}

\begin{proof}
Here we rely on an idea that goes back to Masuda \cite{Mas80} and Angenent \cite{AN90,An90b} to 
introduce parameters and use the implicit function theorem to obtain regularity results for 
solutions, see also \cite{ES96}.
The technical details are included in the appendix.
\end{proof}
\begin{remark}
The preceding results can be slightly weakened to allow for arbitrary
values of $\mu \in (1/2, 1],$ i.e. without eliminating the possibility 
that $4 \mu + \alpha \in \mathbb{Z}$, by taking initial data from the 
continuous interpolation spaces $(E_0, E_1)_{\mu, \infty}^0$, which 
coincide with the Zygmund spaces over $\Tone$.
\end{remark}

\section{Characterizing The Equilibria of ASD}\label{Sec:Equilibria}

We begin our analysis of the long-time behavior of solutions by characterizing
and describing the equilibria of \eqref{ASD}. For this characterization, we make use of a well-known,
strict Lyapunov functional for the surface diffusion flow, namely the surface area functional, and
a characterization of surfaces of revolution with prescribed mean curvature, as presented by Kenmotsu
\cite{KEN80}.

Recalling the operator $G$, as expressed by \eqref{Eqn:GDefined} and taking it to be 
defined on $V_1 \subset h^{4 + \alpha}(\Tone)$, one will see that the set of equilibria of \eqref{ASD} 
coincides with the null set of $G$. Although, from the well-posedness
results of the previous section, we know that we can consider \eqref{ASD} with initial
conditions in $h^{2 + \alpha}(\Tone)$, upon which the operator $G$ is not defined, 
one immediately sees that all equilibria must be in $h^{4 + \alpha}(\Tone)$
(in fact, by Proposition~\ref{Prop:SolutionRegularity}, we can conclude that equilibria 
are in $C^{\omega}(\Tone)$). More specifically, 
if we define equilibria to be those elements $\bar{r} \in V_{1/2} = V \cap h^{2 + \alpha}(\Tone)$, 
such that the maximal solution $r(\cdot, \bar{r})$ satisfies 
\[
r(t, \bar{r}) = \bar{r}, \qquad t > 0,
\] 
then it follows immediately that $\bar{r} \in h^{4 + \alpha}(\Tone)$ and $G(\bar{r}) = 0$.
Now, we proceed by characterizing the elements of the null set of $G$.

Consider the functional
\[
S(r) := \int_{\Tone} r(x) \sqrt{1 + r^2_x(x)} dx,
\]
which corresponds to the surface area of $\Gamma(r)$.
If $r = r(\cdot, r_0)$ is a solution to \eqref{ASD} on the interval $J(r_0)$, 
then (suppressing the variable of integration) 
\begin{align*}
\partial_t S(r(t)) &= \int_{\Tone} \left[ \sqrt{1 + r^2_x(t)} + \frac{r(t) r_{x}(t)}{\sqrt{1 + r^2_x(t)}} \; \partial_x \right] G(r(t)) \; dx\\
&= \int_{\Tone} \partial_x \left( \frac{r(t)}{\sqrt{1 + r^2_x(t)}} \; \partial_x \mathcal{H}(r(t)) \right) \, \mathcal{H}(r(t)) \; dx\\
&= - \int_{\Tone} \frac{r(t)}{\sqrt{1 + r^2_x(t)}} \left( \partial_x \mathcal{H}(r(t)) \right)^2 \; dx, \qquad \qquad t \in J(r_0) \setminus \{ 0 \},
\end{align*}
where we use integration by parts twice and eliminate boundary terms because of periodicity.
Notice that the expression is non-positive for all times $t \in J(r_0) \setminus \{ 0 \}$.  

If $\bar{r}$ is an equilibrium of \eqref{ASD} it follows that $\partial_x \mathcal{H}(\bar{r})$ 
is identically zero on $\Tone$. 
Meanwhile, by definition of the operator $G$, $G(\bar{r}) = 0$ whenever 
$\partial_x \mathcal{H}(\bar{r}) = 0$. Hence, we conclude that $S(r)$ is a strict 
Lyapunov functional for \eqref{ASD}, as claimed, and we also see that the equilibria of 
\eqref{ASD} are exactly those functions $\bar{r} \in h^{4 + \alpha}(\Tone)$ for which 
the mean curvature function $\mathcal{H}(\bar{r})$ is constant on $\Tone$. 

The axisymmetric surfaces with constant mean curvature have been 
characterized explicitly by Kenmotsu in \cite{KEN80}. All equilibria 
of \eqref{ASD} are so-called \emph{undulary} curves, and the \emph{unduloid} 
surfaces, which are generated by the undulary curves by revolution about the axis of symmetry, 
are stationary solutions of the original surface diffusion problem \eqref{SD}.
\\
\goodbreak
\begin{thm}[Delaunay \cite{DE41} and Kenmotsu \cite{KEN80}]\label{Thm:Equilibria}
Any complete surface of revolution with constant mean curvature $\mathcal{H}$ is either a sphere,
a catenoid, or a surface whose profile curve is given (up to translation along the axis of symmetry)
by the parametric expression, parametrized by the arc-length parameter $s \in \mathbb{R}$,
\begin{equation}\label{Eqn:Unduloids}
R(s; \mathcal{H}, B) := \Bigg( \int_{\pi / 2 \mathcal{H}}^s \frac{1 + B \sin ( \mathcal{H} t)}{\sqrt{1 + B^2 + 2 B \sin ( \mathcal{H} t)}} \; dt \, ,
\frac{\sqrt{1 + B^2 + 2 B \sin ( \mathcal{H} s)}}{ \, |\mathcal{H}|} \Bigg).
\end{equation}
\end{thm}

\medskip

\begin{rem}\label{Rem:Equilibria}
We can immediately draw several conclusions from Theorem~\ref{Thm:Equilibria} and characterize the equilibria
of \eqref{ASD}. We use the notation $R(\mathcal{H}, B)$ to denote the curve in $\mathbb{R}^2$ with 
parametric expression $R(\cdot \, ; \mathcal{H}, B)$.
\begin{enumerate}

\item Although the curves $R(\mathcal{H}, B)$ are well-defined for arbitrary 
values $B \in \mathbb{R}$ and $\mathcal{H} \not= 0$, it is not difficult to see that,
up to translations along the $x$--axis, we may restrict our attention to values $\mathcal{H} > 0$ and
$B \geq 0$, c.f. \cite[Section~2]{KEN80}. However, in the sequel we will consider the unduloids in the
setting of even functions on $\Tone$, for which we will benefit by allowing $B < 0$. 

\item When $|B| = 1$, $R(\mathcal{H}, B)$ corresponds to a family of spheres controlled by the parameter $\mathcal{H}$.
The spheres are a well-known family of stable equilibria for the surface diffusion flow,
c.f. \cite{EMS98, Wh12}, however their profile curves are outside of our current setting because they fail to be continuously 
differentiable functions on all of $\Tone$. Moreover, we note that the spheres represented by $R(\mathcal{H}, \pm 1)$
are in fact a {\em connected} family of spheres, or a \emph{chain of pearls} (see Figure~\ref{Fig:Spheres})\footnote{All 
figures were generated with GNU Octave, 
version 3.4.3, copyright 2011 John W. Eaton, and GNUPLOT, version 4.4 patchlevel 3, copyright 2010 Thomas Williams,
Colin Kelley.}, for which even general 
techniques for \eqref{SD} break down, as the manifold is singular at the points of intersection.
These families of connected spheres may be interesting objects to 
investigate in a weaker formulation of ASD, but they fall outside of the current setting.

\item Catenoids, or more precisely the generating catenary curves (which are essentially just the 
hyperbolic cosine, up to scaling), 
fail to satisfy periodic boundary conditions, c.f. Figure~\ref{Fig:Spheres}.

\item In case $|B| > 1$, the curve $R(\mathcal{H}, B)$ is called a \emph{nodary} (see Figure~\ref{Fig:Nodoids}),
which cannot be realized as the graph of a function over the $x$-axis and hence falls outside the current setting. 
A reformulation of \eqref{ASD} to allow for immersed surfaces would permit nodary curves as equilibria.
Such an extension of the current setting may prove beneficial to the investigation of pinch--off,
as it may likely be easier to handle concerns regarding concentration of curvature for solutions near
nodary curves, rather than embedded undularies.

\item For values $0 \leq |B| < 1$, $R(\mathcal{H}, B)$ is the family of \emph{undulary} curves,
which generate the \emph{unduloid} surfaces. The undulary curves are representable as graphs of 
functions over the $x$-axis, which are strictly positive for $B$ in the given range (see Figure~\ref{Fig:Unduloids}). 
In fact, the 
case $B = 0$ corresponds to the cylinder of radius $1 / \mathcal{H}$. Hence, by Theorem~\ref{Thm:Equilibria}
above, we conclude that {\bfseries all equilibria of \eqref{ASD} fall into the family of undulary curves.}

\item Notice that the curve $R(\mathcal{H},B)$ 
is always periodic in both the parameter $s$ and the spacial variable $x$. 
In order to ensure that the curve satisfies the $2 \pi$-periodic boundary
conditions enforced in \eqref{ASD} (which we emphasize is a condition regarding periodicity over
the variable $x$ and not the arc-length parameter $s$), 
we must impose further conditions on the parameters $\mathcal{H}$ 
and $B$; here we avoid $B = 0$ because the curve $R(\mathcal{H}, 0)$ 
trivially satisfies periodic boundary conditions.
In particular, for $B \not= 0$, if $\mathcal{H}$ and $B$ satisfy the relationship
\begin{equation}\label{Eqn:HandBRelation}
\frac{\pi \, \mathcal{H}}{k} = \int_{\pi/2}^{3 \pi/2} \frac{1 + B \sin t}{\sqrt{1 + B^2 + 2B \sin t}} \; dt \,,
\end{equation}
then the curve $R(\mathcal{H}, B)$ is $2 \pi / k$ periodic in the $x$ variable,
for $k \in \mathbb{N}$.
In the sequel, we will use the notation $R(B,k)$ to denote the $2 \pi/k$ periodic
undulary curve with free parameter $-1 < B < 1$ and parameter $\mathcal{H} = \mathcal{H}(B)$ 
fixed according to \eqref{Eqn:HandBRelation}.

\item 
The role of the parameters $B$ and $k$ is clearly seen in the context of 
Delaunay's construction. By rolling an ellipse with eccentricity
$B$ along the $x$--axis, the path traced out by one focus is an undulary curve.
Here $B < 0$ corresponds to a reassignment of major and minor axes in the associated ellipse.
Further, it is clear that the ellipses are restricted to those with  
circumference $2 \pi / k$, to match periodic boundary conditions.

\end{enumerate}
\end{rem}

\begin{figure}[ht]
\centering
\mbox{\subfigure{\includegraphics[width=2.4in,height=1.5in,clip=true,trim=.5in 2.5in .5in 2.5in]{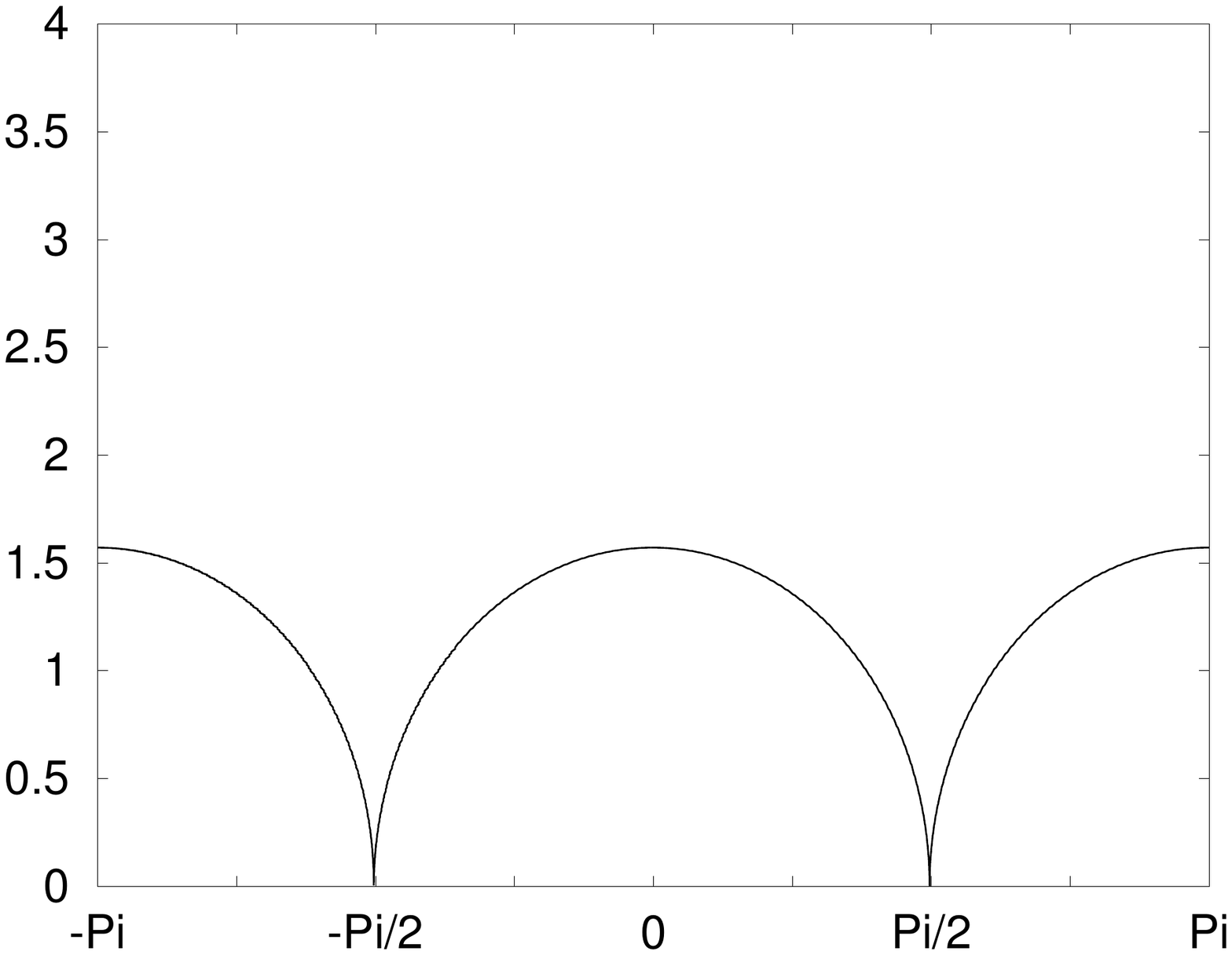}}\quad
\subfigure{\includegraphics[width=2.4in,height=1.5in,clip=true,trim=.5in 2.5in .5in 2.5in]{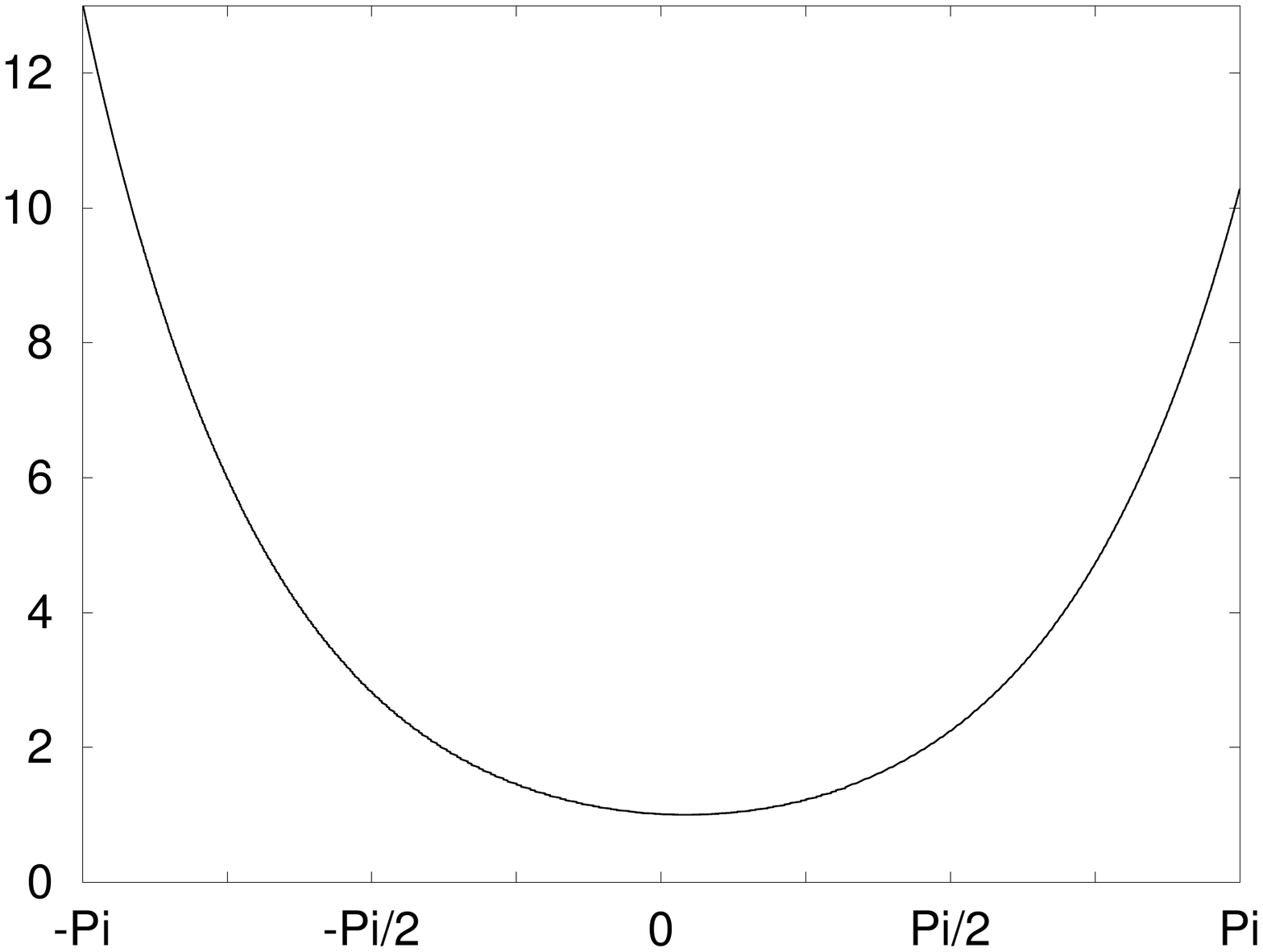} }}
\caption[Periodic Family of Spheres and a Catenary Curve]{Profile curves for a family of spheres and a catenoid, respectively.} \label{Fig:Spheres}
\end{figure}


\begin{figure}[ht]
\centering
\mbox{\subfigure{\includegraphics[width=2.4in,height=1.5in,clip=true,trim=.5in 2.5in .5in 2.5in]{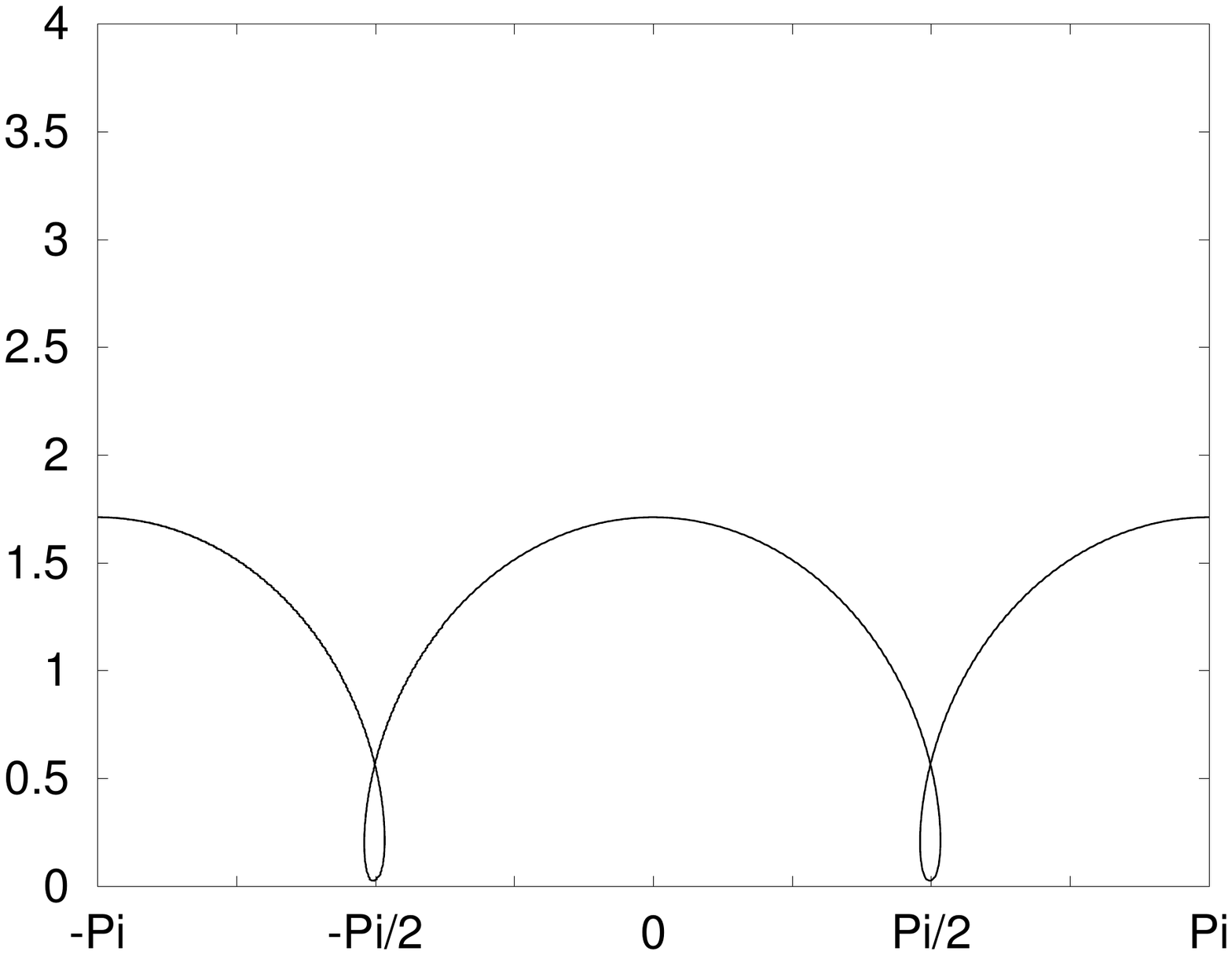}}\quad
\subfigure{\includegraphics[width=2.4in,height=1.5in,clip=true,trim=.5in 2.5in .5in 2.5in]{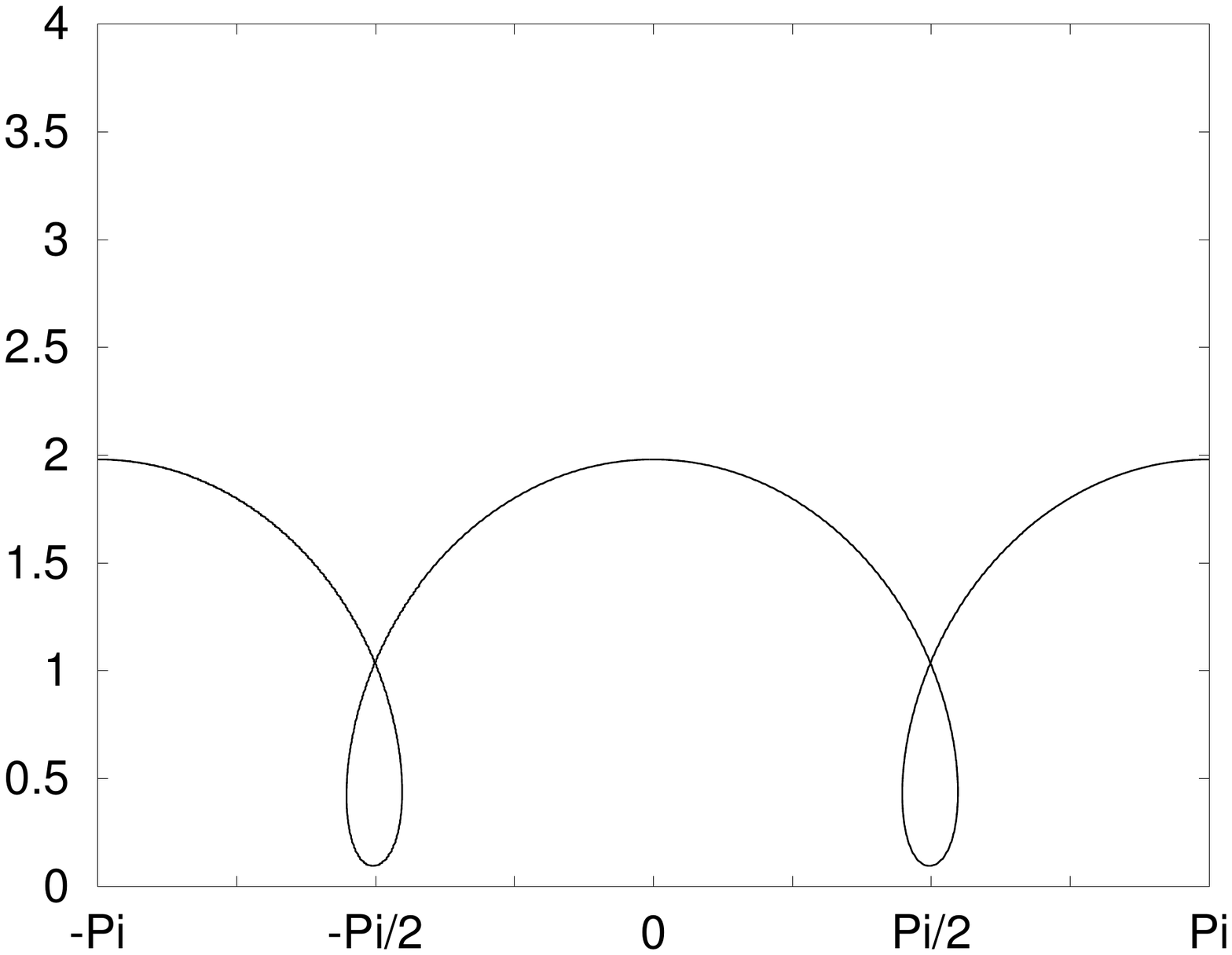} }}
\caption[Nodary Curves]{$\pi$ periodic nodary curves with $B = 1.03$ and $B = 1.1$, respectively.} \label{Fig:Nodoids}
\end{figure}

\pagebreak

\begin{figure}[ht]
\centering
\mbox{\subfigure{\includegraphics[width=2.4in,height=1.5in,clip=true,trim=.5in 2.5in .5in 2.5in]{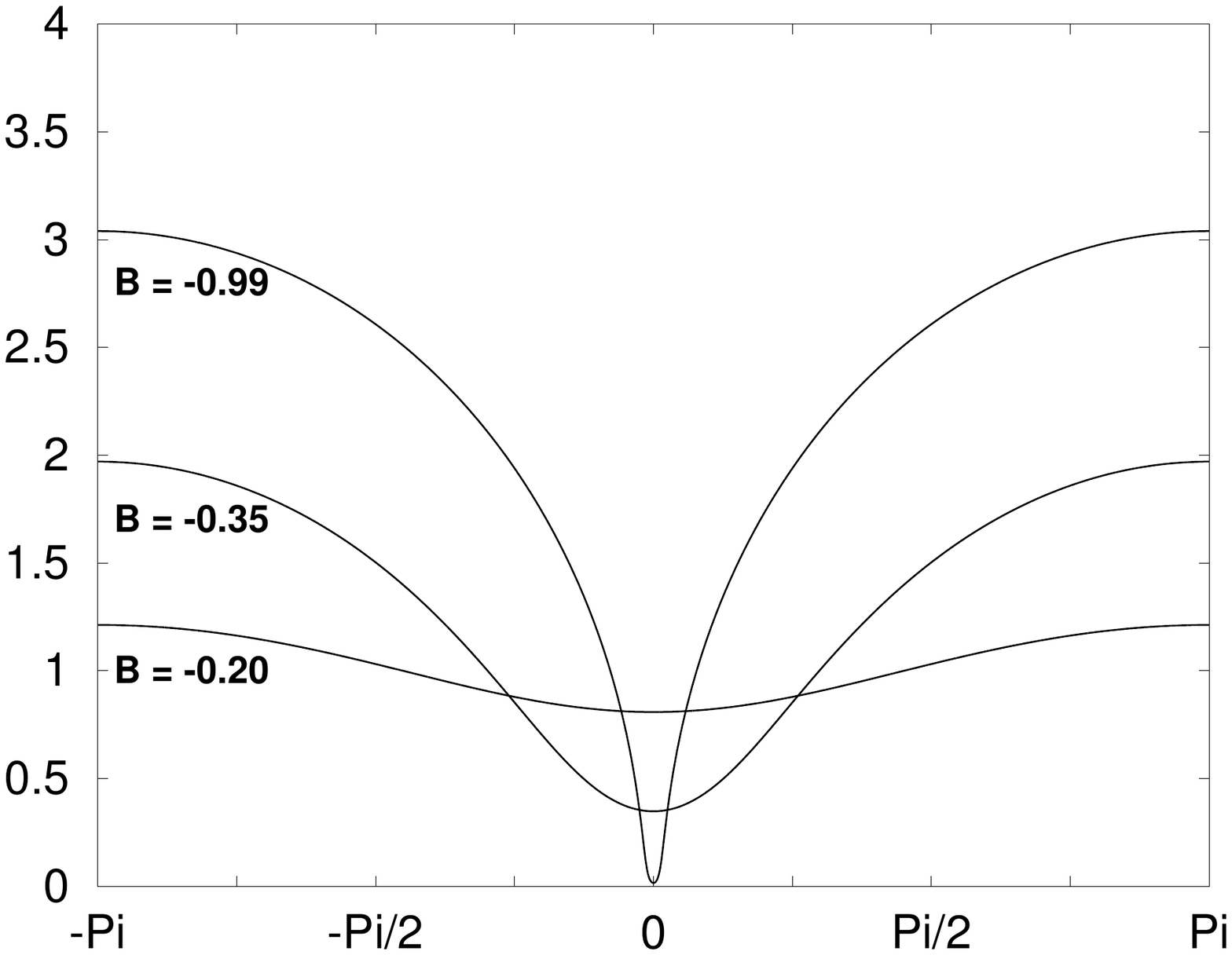}}\quad
\subfigure{\includegraphics[width=2.4in,height=1.5in,clip=true,trim=.5in 2.5in .5in 2.5in]{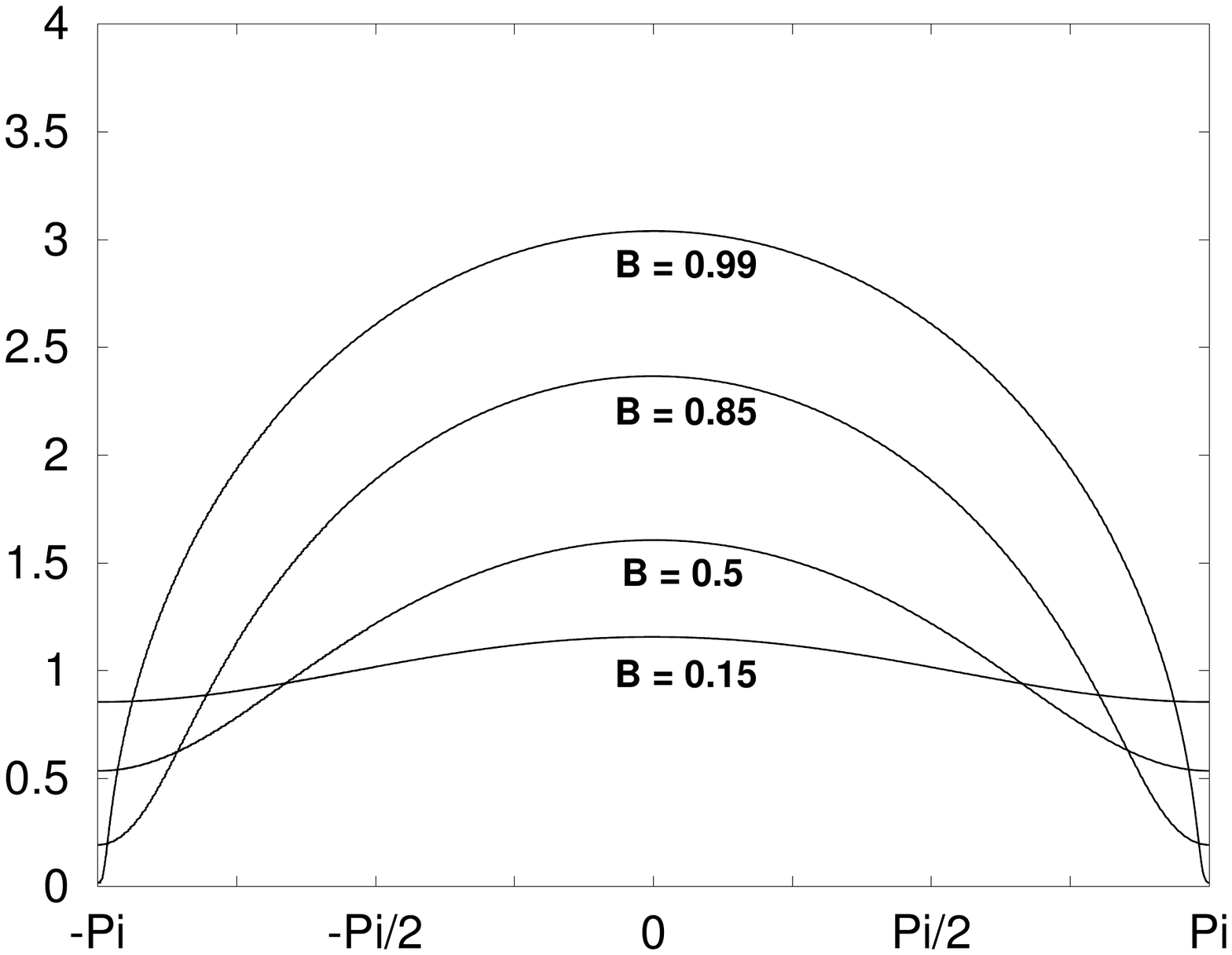} }}
\caption[Undulary Curves]{Families of $2 \pi$ periodic undulary curves with selected parameter values from $B = -.99$
to $B = 0.99$, as indicated.} \label{Fig:Unduloids}
\end{figure}

\section{Stability Of Cylinders With Large Radius}\label{Sec:Stability}

As seen above, the constant function $r(x) \equiv \rs$, for $\rs > 0$, is an
equilibrium of \eqref{ASD2}. Moreover, the constant function $r(x) \equiv \rs$ is associated
to the cylinder $\Gamma(\rs)$ with radius $\rs$, which is a stationary solution of the original 
surface diffusion problem \eqref{SD}. In this section, we establish tools for and carry out the 
investigation of nonlinear stability for these equilibria. 


\subsection{Preliminary Analysis and Definitions}\label{Section:Definitions}

Throughout this analysis, we consider an arbitrary $\rs > 0$ and $\sigma \in \mathbb{R}_+ \setminus \mathbb{Z}$,
unless otherwise stated. Focusing on the properties of solutions near $\rs$, we shift our equations, 
including the \emph{shifted} operator 
\begin{align*}
G_{\star}(\rho) := G(\rho + \rs) &= \frac{1}{\rho + \rs} \; \partial_x \left[ \frac{\rho + \rs}{\sqrt{1 + \rho_x^2}} \; \partial_x  
\mathcal{H}(\rho + \rs) \right],
\end{align*}
which maps $\rho \in E_1 \cap \, U_{\star}$ to $E_0$, where we consider $\rho = r - \rs$, 
and is in the regularity class
$C^{\omega}$ by Lemma~\ref{Lem:AandFReg}; here we take $U_{\star} := V - \rs := \{ \rho - \rs: \rho \in V \}$.
Now we consider the surface diffusion problem shifted by $\rs$,
\begin{equation}\label{ASDrs}
\begin{cases}
\rho_t(t,x) = G_{\star}(\rho(t,x)), &\text{$x \in \Tone, \; t > 0$},\\
\rho(0,x) = \rho_0(x), &\text{$x \in \Tone$},
\end{cases}
\end{equation}
where $\rho_0 := r_0 - \rs$. We say that 
\[
\rho = \rho(\cdot, \rho_0) \in C^1(\dot{J}, E_0) \cap C(\dot{J}, E_1) \cap C(J, E_{\mu} \cap U_{\star})
\]
is a solution to \eqref{ASDrs}, with initial data $\rho_0 \in E_{\mu} \cap U_{\star}$, 
on the interval $J \subset \mathbb{R}_+$ if $\rho$ satisfies \eqref{ASDrs} pointwise, for $t > 0$, and 
$\rho(0) = \rho_0$. We investigate the properties of $G_{\star}$ around 0 
in order to gain information about the stability of $\rs$ in \eqref{ASD}.

Define the functional 
\[
F_{\star}(\rho) = F_{\star}(\rho; \rs) := \int_{\Tone} \big( \rho(x) + \rs \big)^2 dx,
\]
which corresponds to the volume enclosed by the surface $\Gamma(\rho + \rs)$. 
It follows from the analyticity of multiplication and integration
on little-H\"older spaces that $F_{\star}$ is of class $C^{\omega}$ from $h^{\sigma}(\Tone)$ to $\mathbb{R}, \;
\sigma \in \mathbb{R}_+ \setminus \mathbb{Z}$. 
The Fr\'echet derivative of $F_{\star}$ is
\begin{equation}\label{Eqn:DF}
DF_{\star}(\rho): \left[ h \longmapsto 2 \int_{\Tone} \big( \rho(x) + \rs \big) h(x) dx \right] \in \mathcal{L} \left( h^{\sigma}(\Tone), \mathbb{R} \right), \qquad \rho \in h^{\sigma}(\Tone, \mathbb{R}).
\end{equation}
Moreover, it holds that $F_{\star}(\rho)$ is conserved along solutions to \eqref{ASDrs}.
Indeed, if $\rho = \rho(\cdot, \rho_0)$ is a solution to \eqref{ASDrs}, then
\begin{align*}
\frac{1}{2}\frac{d}{dt} F_{\star}(\rho(t)) = \int_{\Tone} \big( \rho(t) + \rs \big) \rho_t(t) dx
&= \int_{\Tone} \partial_x \left[ \frac{\big( \rho(t) + \rs \big)}{\sqrt{1 + \rho^2_x(t)}} \; \partial_x \mathcal{H}(\rho(t) + \rs) \right] dx 
= 0, 
\end{align*}
for $t \in J(\rho_0) \setminus \{ 0 \},$ where the last equality holds by periodicity. Thus,
conservation of $F_{\star}$ along the solution $\rho$ follows by continuity of $F_{\star}$ and 
convergence of $\rho$ to the initial data $\rho_0$ in $E_{\mu}$.
From these properties, it follows that 
\begin{equation}\label{Eqn:Manifolds}
\mathcal{M}^{\sigma}_{\eta} := \Big\{ \rho \in h^{\sigma}(\Tone): F_{\star}(\rho) = F_{\star}(\eta) \Big\}, \qquad \eta \in \mathbb{R}, \; \sigma \in \mathbb{R}_+ \setminus \mathbb{Z}
\end{equation}
is a family of invariant level sets for \eqref{ASDrs}. The following techniques are
motivated by results of Prokert \cite{PRO97} and Vondenhoff \cite{VON08}, whereby 
one can take advantage of invariant manifolds in order to derive stability results.

First, we introduce the mapping
\[
P_0 \rho := \rho - \frac{1}{2 \pi}\int_{\Tone} \rho(x) dx,
\]
which defines a projection on $h^{\sigma}(\Tone)$.
We denote by $h^{\sigma}_0(\Tone)$ the image $P_0 \big( h^{\sigma}(\Tone) \big),$ which 
exactly coincides with the zero-mean functions on $\Tone$ in the regularity class
$h^{\sigma}(\Tone)$, and we have the topological decomposition
\[
h^{\sigma}(\Tone) = h^{\sigma}_0(\Tone) \oplus (1 - P_0)\big( h^{\sigma}(\Tone) \big) 
\cong h^{\sigma}_0(\Tone) \oplus \mathbb{R} \,.
\]
In what follows, we equate the constant function 
$[\eta(x) \equiv \eta] \in (1 - P_0) \big( h^{\sigma}(\Tone) \big)$
with the value $\eta \in \mathbb{R},$ and we denote each simply as $\eta$.

Consider the operator
\[
\Phi(\rho, \tilde{\rho}, \eta) := \Big( P_0 \rho - \tilde{\rho}, \; F_{\star}(\rho) - F_{\star}(\eta) \Big),
\]
which maps $h^{\sigma}(\Tone) \times h^{\sigma}_0(\Tone) \times \mathbb{R}$ to 
$h^{\sigma}_0(\Tone) \times \mathbb{R}$ and is of class $C^{\omega}$, by regularity of the mappings $F_{\star}$ and $P_0$.
Notice that $\Phi(0,0,0) = (0,0)$ and, using \eqref{Eqn:DF},
\begin{equation}\label{Eqn:D1Phi}
D_1 \Phi (0,0,0) = 
\Big( P_0, \; 4 \pi \rs (1 - P_0) \Big) \in \mathcal{L}_{isom}(h^{\sigma}(\Tone), h^{\sigma}_0(\Tone) \times \mathbb{R}),
\end{equation} 
i.e. the Fr\'echet derivative of $\Phi$ with respect to the first variable, at the origin, is a linear isomorphism.
Hence, it follows from the implicit function theorem
that there exist neighborhoods $(0,0) \in U = U_0 \times U_1 \subset h^{\sigma}_0(\Tone) \times \mathbb{R}$ 
and $0 \in U_2 \subset h^{\sigma}(\Tone)$
and a $C^{\omega}$ function $\psi: U \rightarrow U_2$ such that, for all $(\rho, \tilde{\rho}, \eta) \in U_2 \times U$,
\[
\Phi(\rho, \tilde{\rho}, \eta) = ( 0, 0 ) \qquad \text{if and only if} \qquad \rho = \psi(\tilde{\rho}, \eta).
\]

\medskip

\begin{rem}\label{Rem:PsiProperties}
We can immediately state the following properties of $\psi$, which follow directly from its definition 
and elucidate the relationship between $P_0$ and $\psi$.
\begin{enumerate}

\item $P_0 \psi(\trho, \eta) = \trho$ for all $(\trho, \eta) \in U$.

\item Given $\rho \in \psi(U) \cap \mathcal{M}^{\sigma}_{\eta}$, it follows that $\psi(P_0 \rho, \eta) = \rho \, .$ 

\item $\psi (0, \eta) = \eta$, for $\eta \in U_1$.
This and the preceding remark follow from the fact that $F_{\star}(\eta)$
is injective when restricted to $\eta \in (-\rs, \infty) \subset \mathbb{R}$.

\item It follows from the identity $\Phi( \psi(\trho, \eta), \trho, \eta) = (0,0)$ and differentiating with
respect to $\trho$ that $D_1 \Phi(\psi(0,\eta), 0, \eta) D_1 \psi(0,\eta) h - (h, 0) = (0,0)$. From this 
observation, and the fact that $D_1 \Phi(\eta, 0, \eta) = (P_0, 4 \pi (\rs + \eta)(1 - P_0))$, it follows
that 
\[
D_1 \psi(0,\eta)h = h, \qquad h \in h^{\sigma}_0(\Tone), \; \eta \in U_1.
\]

\item $\psi(U_0, \eta) \subset \mathcal{M}^{\sigma}_{\eta}$ for $\eta \in U_1$. Hence, 
$\psi(\cdot, \eta)$ can be taken as a (local) parametrization of $\mathcal{M}^{\sigma}_{\eta} \, .$ Moreover,
from the preceding remark and the bijectivity of $\psi(\cdot, \eta)$ from $U_0$ to $\mathcal{M}^{\sigma}_{\eta} \cap U_2$,
we can see that $\mathcal{M}^{\sigma}_{\eta} \cap U_2$ is 
a Banach manifold over $h^{\sigma}_0(\Tone)$ anchored at the point $\eta \in \mathbb{R} \, .$ 

\item For $(\trho, \eta) \in U$, we have the representation
\[
\psi(\trho, \eta) = \Big(P_0 + (1 - P_0) \Big) \psi(\trho, \eta) = \trho + \frac{1}{2 \pi} \int_{\Tone} \psi(\trho, \eta)(x) dx,
\]
and so we can see that $\mathcal{M}^{\sigma}_{\eta} \cap U_2$ can be realized (locally) as the graph of  
a $\mathbb{R}$-valued analytic function over the zero-mean functions $\trho \in h^{\sigma}_0(\Tone)$.  

\item Although $\psi(\cdot, \eta)$ depends upon the parameter $\sigma$, a priori, it follows easily from the 
preceding representation that
\[
\psi(\cdot, \eta): U_0 \cap h^{\tilde{\sigma}}_0(\Tone) \rightarrow h^{\tilde{\sigma}}(\Tone), \quad \tilde{\sigma} \in \mathbb{R}_+ \setminus \mathbb{Z},
\]
so that $\psi$ preserves the spacial regularity of functions regardless of the regularity parameter $\sigma$
with which $\psi$ was constructed. However, notice that the neighborhood $U_0$ will remain intrinsically linked
with the parameter which was used to construct $\psi$.

\end{enumerate}
\end{rem}

With the established invariance and local structure of the sets $\mathcal{M}_{\eta}^{\sigma}$, it follows
that the dynamics governing solutions to \eqref{ASD} reside in the tangent space to the manifold 
$\mathcal{M}_{\eta}^{\sigma} \cap U_2$. Hence, if we reduce \eqref{ASD} to a local system {\bf on} 
$\mathcal{M}_{\eta}^{\sigma} \cap U_2$, then we will have captured all of the dynamics of the problem. 
Remarks~\ref{Rem:PsiProperties}(d)
is the first observation toward this reduced formulation. In fact, one can make use of the properties
established in Remarks~\ref{Rem:PsiProperties} to prove the following, even more general, result
regarding the properties of the the tangent vectors to $\mathcal{M}_{\eta}^{\sigma}$. Although we
use other tools to connect the reduced problem \eqref{PASD} below with the full problem \eqref{ASD}, 
this remark provides good intuition into the nature of these manifolds.

\medskip

\begin{remark}\label{Rem:Tangents}
Given $(\trho, \eta) \in U$ it follows that $D_1 \psi(\trho, \eta) \circ P_0 = id_{T_{\psi(\trho, \eta)}\mathcal{M}^{\sigma}_{\eta}}$, where 
$T_{\rho}\mathcal{M}^{\sigma}_{\eta}$ denotes the tangent space to the manifold $\mathcal{M}^{\sigma}_{\eta}$ at the point $\rho$.
\end{remark}


\subsection{The Reduced Problem}\label{Section:ReducedProblem}
Fix $\alpha \in (0, 1)$ and we denote the spaces
\[
F_0 := h^{\alpha}_0(\Tone), \quad F_1 := h^{4 + \alpha}_0(\Tone), \quad \text{and} \quad
F_{\mu} := (F_0, F_1)_{\mu, \infty}^0, \quad \mu \in (0,1),
\]
so that $F_{\mu} = P_0 E_{\mu}$ for $\mu \in [0,1]$. Define the operator
\[
\mathcal{G}_{\star}(\trho,\eta) = \mathcal{G}_{\star}(\trho,\eta; \rs) := P_0 \, G \big( \psi(\trho, \eta) + \rs \big),
\]
which is defined for all $(\trho, \eta) \in U \subset F_0 \times \mathbb{R}$ 
with $\trho \in U_0 \cap F_1$.

Now we consider the \emph{reduced} problem for the zero-mean functions 
\begin{equation}\label{PASD}
\begin{cases}
\trho_t(t,x) = \mathcal{G}_{\star}(\trho(t,x), \eta), &\text{$t > 0, \; x \in \Tone$},\\
\trho(0,x) = \trho_0(x), &\text{$x \in \Tone$},
\end{cases}
\end{equation}
where $\trho_0 := P_0 r_0 = P_0(r_0 - \rs)$.
One will note that we should insist on $\psi(\trho, \eta)(x) > - \rs$ for all $x \in \Tone$ in order to 
guarantee that $G(\psi(\trho, \eta) + \rs)$ is well-defined. However, we can assume, without loss of
generality, that the neighborhood $U$ is chosen small enough to ensure this property holds for all $(\trho, \eta) \in U$.

\medskip

\begin{rem}
Throughout most of the analysis that follows, we will treat the parameter $\eta$ as a free parameter, although it 
has a very specific interpretation in relation to \eqref{ASD2}. If one is given initial data
$r_0$ close to $\rs$, then the parameter $\eta$ is chosen so that 
\[
F_{\star}(\eta) = F_{\star}(r_0 - \rs) \, .
\]
\begin{enumerate}

	\item Essentially, this parameter allows for the possibility that the volume enclosed by the surface
$\Gamma(r_0)$ differs from that of the cylinder $\Gamma(\rs)$, thereby allowing us to handle
non-volume-preserving perturbations $r_0$ of the cylinder $\rs$. 

	\item From a more general viewpoint, one can see that the family $\{\mathcal{M}^{\sigma}_{\eta} \cap \psi(U): \eta \in U_1 \}$ 
forms a dimension 1 foliation of a neighborhood of the 
positive real axis $\mathbb{R}_+ \subset h^{\sigma}(\Tone)$ and the parameter $\eta$ separates the leaves of the foliation.

\end{enumerate}
\end{rem}

For $\mu \in (0, 1]$ and closed intervals $J \subseteq \mathbb{R}_+$ with $0 \in J$, define the spaces
\begin{align*}
\mathbb{E}_0(J) &:= BU\!C_{1 - \mu}(J, E_0),\\
\mathbb{E}_1(J) &:= BU\!C_{1 - \mu}^1(J, E_0) \cap BU\!C_{1 - \mu}(J, E_1),
\end{align*}
and 
\begin{align*}
\mathbb{F}_{0}(J) &:= BU\!C_{1 - \mu}(J, F_0),\\
\mathbb{F}_{1}(J) &:= BU\!C_{1 - \mu}^1(J, F_0) \cap BU\!C_{1 - \mu}(J, F_1),
\end{align*}
within which we will discuss solutions to the shifted problem \eqref{ASDrs} and 
the reduced problem \eqref{PASD}, respectively. 

In order to connect these two problems, we will make use of the \emph{lifting} map $\psi$, 
defined in the previous section. To ensure that $\psi$ is well-defined on $\mathbb{F}_1(J)$,
we must restrict our attention to functions which map into an appropriate neighborhood 
$U_0 \subset F_0$ of 0. In particular, we assume that $U_0$ is given 
so that 
\[
\psi(\cdot, \eta): U_0 \subset F_0 \rightarrow E_0, \quad \eta \in U_1,
\]
is in the regularity class $C^{\omega}$ and, without loss of generality, we assume that
$U_0$ is given sufficiently small so that $\psi$ and the derivative $D_1 \psi$ 
are bounded on $U = U_0 \times U_1$. More precisely, $U_0$ is chosen 
sufficiently small so that there exists a constant $N > 0$ for which the inequalities
\begin{equation}\label{Eqn:PsiBound}
\| \psi(\trho, \eta) \|_{E_0} \leq N \quad \text{and} \quad \| D_1 \psi(\trho, \eta) \|_{\mathcal{L}(F_0, E_0)} \leq N 
\end{equation}
hold for all $(\trho, \eta) \in U = U_0 \times U_1.$

\begin{lemma}\label{Lem:PsiLifts}
Fix $\eta \in U_1$ and $J := [0, T]$ for $T > 0$. Then
\[
\psi(\cdot, \eta) : \mathbb{F}_1(J) \cap C(J, U_0) \longrightarrow \mathbb{E}_1(J), \qquad \text{with} \quad \psi(\trho, \eta)(t) := \psi(\trho(t), \eta).
\]
Moreover, if $\trho_0 \in F_{\mu}$ and $\trho = \trho(\cdot, \trho_0) \in \mathbb{F}_1(J) \cap C(J, U_0)$ 
is a solution to \eqref{PASD}, for some $\mu \in [1/2, 1]$, then $\rho := \psi(\trho, \eta)$ 
is the unique solution on the interval $J$ to \eqref{ASDrs}, with initial data 
$\rho_0 := \psi(\trho_0, \eta) \in E_{\mu}$.  
\end{lemma}

\begin{proof}
First notice that the embeddings
\begin{equation}\label{Eqn:Embedding}
\mathbb{F}_1(J) \hookrightarrow BU\!C(J, F_{\mu}) \hookrightarrow BU\!C(J, F_0), \qquad \mu \in [1/2, 1],
\end{equation}
follow from \cite[Theorem III.2.3.3]{AM95} and the continuous embedding of little-H\"older 
spaces, respectively.

To see that the mapping property for $\psi(\cdot, \eta)$ holds, let $\trho \in \mathbb{F}_1(J) \cap C(J,U_0)$. 
Uniform continuity and differentiability of the function $\psi(\trho(\cdot), \eta)$ follows
from the regularity of $\psi$ and $\trho$, and compactness of the interval $J$.
Hence we focus on demonstrating that $\psi(\trho(\cdot), \eta)$ satisfies the boundedness conditions for 
$\mathbb{E}_1(J)$. 
In the case $\mu \in [1/2, 1)$, it follows from Remarks~\ref{Rem:PsiProperties}(f) and \eqref{Eqn:PsiBound}
that, for $t \in \dot{J}$, 
\begin{equation}\label{Eqn:Bound1}
\begin{split}
t^{1 - \mu} \| \psi(\trho(t), \eta) \|_{E_1} 
&\leq t^{1 - \mu} \| \trho(t) \|_{F_1} + \frac{t^{1 - \mu}}{2 \pi} \int_{\Tone} | \psi(\trho(t), \eta)(x) | dx\\ 
&\leq \| \trho \|_{\mathbb{F}_1(J)} + t^{1 - \mu} \| \psi(\trho(t), \eta) \|_{C(\Tone)}\\
& \leq \| \trho \|_{\mathbb{F}_1(J)} + T^{1 - \mu} N, \\
\text{and} \qquad \lim_{t \rightarrow 0} t^{1 - \mu} &\| \psi(\trho(t), \eta) \|_{E_1} = 0.
\end{split}
\end{equation}
From \eqref{Eqn:Bound1} we conclude that $\psi(\trho, \eta) \in BU\!C_{1 - \mu}(J, E_1).$
Meanwhile, looking at the time derivative of $\psi(\trho,\eta)$, we note that 
$\partial_t \psi(\trho(t), \eta) = D_1 \psi(\trho(t), \eta) \partial_t \trho(t)$ and so we again make
use of \eqref{Eqn:PsiBound} to see that 
\begin{align*}
t^{1 - \mu} \| \partial_t \psi(\trho(t), \eta) \|_{E_0} 
&\leq \| D_1 \psi(\trho(t), \eta) \|_{\mathcal{L} (F_0, E_0)} t^{1 - \mu} \| \partial_t \trho(t) \|_{F_0}\\ 
&\leq N \| \trho \|_{\mathbb{F}_1(J)} < \infty,\\
\text{and} \qquad \lim_{t \rightarrow 0} t^{1 - \mu} &\| \partial_t \psi(\trho(t), \eta) \|_{E_0} = 0.
\end{align*}
Hence, making use of the embedding $E_1 \hookrightarrow E_0$, we see that 
$\psi(\trho, \eta) \in \mathbb{E}_1(J)$, as desired.
Meanwhile, when $\mu = 1$ we again get continuity and differentiability from the regularity of the 
mappings $\trho$ and $\psi$.

To see that the second part of the lemma holds, observe by \eqref{Eqn:Embedding} that
$\rho_0 := \psi(\trho_0, \eta) \in E_{\mu} \cap U_{\star}$. 
Hence, by Proposition~\ref{Prop:Existence}, there exists a unique maximal solution
\[
r(\cdot, \rho_0) \in C_{1 - \mu}^1(J(\rho_0), E_0) \cap C_{1 - \mu}(J(\rho_0), E_1)
\]
to \eqref{ASDrs} on some maximal interval of existence $J(\rho_0) = [0, t^+(\rho_0))$. 
Now, define $\rho(\cdot) := \psi(\trho(\cdot),\eta)$ as indicated and it suffices to 
show that $\rho_t(t) = G_{\star}(\rho(t))$ for $t \in \dot{J} := (0, T]$, 
since this will imply that $\rho(t) = r(t, \rho_0)$ by uniqueness 
and maximality of the solution $r(\cdot, \rho_0)$. Proceeding, 
let $t \in \dot{J}$ and consider the auxiliary problem  
\[
\begin{cases} \dot{\gamma}(\tau) = G_{\star}(\gamma(\tau)), &\text{for $\tau \in [0, \varepsilon],$}\\
\gamma(0) = \rho(t), \end{cases}
\] 
which has a unique solution $\gamma \in C^1([0, \varepsilon], E_0) \cap C([0, \varepsilon], E_1)$
by Proposition~\ref{Prop:Existence}, provided we choose $\varepsilon > 0$ sufficiently small
for the particular value $\rho(t) \in E_1$. Notice, by the regularity of $\gamma$, we have
\[
\dot{\gamma}(0) = G_{\star}(\gamma(0)) = G_{\star}(\rho(t)).
\]
Further, note that $\rho(t) \in \mathcal{M}_{\eta}^{4 + \alpha}$,
from which we conclude that $\gamma(\tau) \in \mathcal{M}_{\eta}^{4 + \alpha}$ and by 
Remarks~\ref{Rem:PsiProperties} we have the representation $\gamma(\tau) = \psi( P_0 \gamma(\tau), \eta)$,
$\tau \in [0, \varepsilon]$. Finally, we see that
\begin{align}\label{Eqn:GTangent}
G_{\star}(\rho(t)) &= \dot{\gamma}(0) = \partial_{\tau} \left( \psi(P_0 \gamma(\tau), \eta) \right) \Big|_{\tau = 0} 
= D_1 \psi(P_0 \gamma(0), \eta) P_0 \dot{\gamma}(0) \nonumber\\
&= D_1 \psi(P_0 \rho(t), \eta) P_0 G_{\star}( \rho(t)) = D_1 \psi( \trho(t), \eta) \mathcal{G}_{\star}(\trho(t), \eta)\\
&= \partial_t \left( \psi( \trho(t), \eta) \right) = \rho_t(t), \nonumber
\end{align}
which concludes the proof.
\end{proof}

We also get the following results, which further illuminate the relationship between 
the mappings $G_{\star}$ and $\mathcal{G}_{\star}$ and explicitly connect the 
equilibria of the two problems \eqref{ASDrs} and \eqref{PASD}.

\begin{lemma}\label{Lem:GtoG}
For any $\rho \in \mathcal{M}^{4 + \alpha}_{\eta} \cap U_2$, it follows that
\begin{equation}\label{Eqn:LiftG}
G_{\star}(\rho) = D_1 \psi(P_0 \rho, \eta) P_0 G_{\star}(\rho),
\end{equation}
and
\begin{equation}\label{Eqn:DerivativeG}
DG_{\star}(\rho) h = D^2_{1} \psi(P_0 \rho, \eta) [P_0 h, P_0 G_{\star}(\rho)] + D_1 \psi(P_0 \rho, \eta) P_0 DG_{\star}(\rho) h,
\end{equation}
for $h \in E_1$.
\end{lemma}

\begin{proof}
The first claim was justified in the proof of Lemma~\ref{Lem:PsiLifts} above and is expressed in 
\eqref{Eqn:GTangent}. Meanwhile, the second claim follows immediately by differentiation.
\end{proof}

\begin{prop}\label{Prop:Equilibria}
If $(\trho, \eta) \in U$, then $(\trho, \eta)$ is an equilibrium of \eqref{PASD}
if and only if $\psi(\trho, \eta)$ is an equilibrium of \eqref{ASDrs}, i.e. 
\[
\mathcal{G}_{\star}(\trho, \eta) = 0 \quad \Longleftrightarrow \quad G_{\star}(\psi(\trho, \eta)) = 0.
\]
Moreover, if $\mathcal{G}_{\star}(\trho, \eta) = 0$, then it follows that
\begin{equation}\label{Eqn:GLinear1}
D G_{\star}(\psi(\trho, \eta)) h = D_1 \psi(\trho, \eta) P_0 D G_{\star}(\psi(\trho, \eta)) h, \qquad h \in E_1,
\end{equation}
and
\begin{equation}\label{Eqn:GLinear2}
D G_{\star}(\psi(\trho, \eta)) D_1 \psi(\trho, \eta) \tilde{h} = D_1 \psi(\trho, \eta) D_1 \mathcal{G}_{\star}(\trho, \eta) \tilde{h}, \qquad \tilde{h} \in F_1.
\end{equation}
\end{prop}

\begin{proof}
The first claim follows from the definition of $\mathcal{G}_{\star}$ and \eqref{Eqn:LiftG}, while 
\eqref{Eqn:GLinear1} is a consequence of \eqref{Eqn:DerivativeG} and \eqref{Eqn:GLinear2} follows from \eqref{Eqn:DerivativeG} and the chain rule:
\begin{align*}
D G_{\star}(\psi(\trho, \eta)) D_1 \psi(\trho, \eta) \tilde{h} &= D_1 \psi(\trho, \eta) P_0 D G_{\star}(\psi(\trho, \eta)) D_1 \psi(\trho, \eta) \tilde{h}\\ &= D_1 \psi(\trho, \eta) D_1 \mathcal{G}_{\star}(\trho, \eta) \tilde{h}. \qedhere
\end{align*}
\end{proof}


\subsection{Mapping Properties of $D_1 \mathcal{G}_{\star}(0,\eta)$}
 
Notice that the points $(0, \eta) \in U$ are equilibria of \eqref{PASD},
and they correspond to the cylinders $\Gamma(\rs + \eta)$.
We are interested in the spectral properties of the linearization of $\mathcal{G}_{\star}$ 
about these equilibria. In particular, we compute the Fr\'echet derivative
\[
D_1 \mathcal{G}_{\star}(0, \eta) h = P_0 \, DG_{\star}( \psi(0, \eta)) \, D_1 \psi(0, \eta) h = P_0 \, DG_{\star}(\eta) \, D_1 \psi(0, \eta) h \,, 
\]
for $h \in F_1 \,.$
Hence, 
by Remarks~\ref{Rem:PsiProperties}(d)  
we derive the formula 
\begin{equation}\label{Eqn:D1PG}
D_1 \mathcal{G}_{\star}(0, \eta) = P_0 DG_{\star}(\eta) \big|_{F_1} = DG_{\star}(\eta) \big|_{F_1},
\end{equation}
where the last equality is verified by application of the divergence theorem to the linearization
\begin{equation}\label{Eqn:DG}
DG_{\star}(\eta) = - \partial_x^2 \left( \frac{1}{(\rs + \eta)^2} + \partial_x^2 \right).
\end{equation}
Utilizing the Fourier series representation of functions in $h^{\sigma}(\Tone)$,
c.f. \cite[Propositions 1.2 and 1.3]{LeC11}, we find the eigenvalues of this
linearized operator. In particular, for $h \in E_1$,
\begin{align}\label{Eqn:SpectralSet}
(\lambda - DG_{\star}(\eta)) \; h &= \left( \lambda + \partial_x^2 \left( \frac{1}{(\rs + \eta)^2} + \partial_x^2 \right) \right) 
\sum_{k \in \mathbb{Z}} \hat{h}(k) e_k \nonumber\\
&= \sum_{k \in \mathbb{Z}} {\left(\lambda - k^2 \left( \frac{1}{(\rs + \eta)^2} - k^2 \right) \right) \hat{h}(k) e_k} \nonumber\\
\Longrightarrow \qquad \sigma_p(DG_{\star}(\eta)) &= \left\{ k^2 \left( \frac{1}{(\rs + \eta)^2} - k^2 \right): k \in \mathbb{Z} \right\}.
\end{align}

Noting that the embedding $E_1 \hookrightarrow E_0$ is compact,
it follows that the resolvent $R(\lambda) := (\lambda - DG_{\star}(\eta))^{-1}$
is a compact operator, $\lambda$ in the resolvent set
$\rho(DG_{\star}(\eta))$. 
It follows from classic theory of linear operators 
that the spectrum $\sigma(DG_{\star}(\eta))$ consists entirely of isolated 
eigenvalues of finite multiplicity, see Kato \cite[Theorem III.6.29]{KAT76} for instance. Hence,  
$\sigma_p(DG_{\star}(\eta)) = \sigma(DG_{\star}(\eta)).$

\begin{remark}
If $\rs + \eta > 1$, then $\sigma(DG_{\star}(\eta)) \subset (-\infty, 0]$, 
however the spectrum will always contain 0. The presence of this 0 
eigenvalue can be seen as a consequence of the fact that the equilibria $\rs + \eta$ 
are not isolated in the space $E_1$. Hence, by passing to the operator 
$\mathcal{G}_{\star}$, which acts on an open subset of the zero-mean functions 
$F_1$, we eliminate 
the nontrivial equilibria (since the only constant function in $F_1$ 
is the zero function) and thereby eliminate the zero eigenvalue. 
In particular, one easily computes that 
\begin{equation}\label{Eqn:Spectra}
\sigma(D_1 \mathcal{G}_{\star}(0,\eta)) = \left\{ k^2 \left( \frac{1}{(\rs + \eta)^2} - k^2 \right): k \in \mathbb{Z} \setminus \{ 0 \} \right\}, \qquad \eta \in U_1.
\end{equation}
\end{remark}


Before we return to the problem \eqref{ASD}, we state the following maximal regularity 
result for the linearization $D_1 \mathcal{G}_{\star}(0,\eta)$. 
For this result, we define the exponentially weighted maximal regularity spaces
\[
\mathbb{F}_j(\mathbb{R}_+, \omega) := \Big\{ f : (0, \infty) \rightarrow F_0 
\, \Big| \, [t \mapsto e^{\omega t} f(t)] \in \mathbb{F}_j(\mathbb{R}_+) \Big\}, \qquad \omega \in \mathbb{R}, \, j = 0,1,
\]
which are Banach spaces when equipped with the norms
$\| u \|_{\mathbb{F}_j(\mathbb{R}_+, \omega)} := \| e^{\omega t} u \|_{\mathbb{F}_j(\mathbb{R}_+)}.$

\begin{thm}\label{Thm:ProjMaxReg}
Suppose $\rs > 1$ and $\mu \in (0, 1]$. There exist nonzero positive constants 
$\delta = \delta(\rs)$ and 
$\omega = \omega(\rs, \delta)$ such that
\[
\Big( \mathbb{F}_0(\mathbb{R}_+, \omega), \; \mathbb{F}_1(\mathbb{R}_+, \omega) \Big)
\]
is a pair of maximal regularity for $-D_1 \mathcal{G}_{\star}(0,\eta)$,
for any $\eta \in (- \delta, \delta)$. That is,
\[
\left( \partial_t - D_1 \mathcal{G}_{\star}(0,\eta), \gamma \right) \in 
\mathcal{L}_{isom} \left( \mathbb{F}_1(\mathbb{R}_+, \omega), \mathbb{F}_0(\mathbb{R}_+, \omega) \times h^{4 \mu + \alpha}_0(\Tone) \right),
\]
holds uniformly for $\eta \in (- \delta, \delta)$.
\end{thm}

\begin{proof}
Fix $\delta > 0$ so that $(- \delta, \delta) \subset U_1 \cap (1 - \rs,\infty)$. 
Following the notation and definitions of \cite{LeC11}, it is clear from the representation \eqref{Eqn:DG} 
that $-DG_{\star}(\eta)$ is a uniformly elliptic operator from which we see, by \cite[Theorem 4.4]{LeC11}, 
that $DG_{\star}(\eta)$ generates an analytic semigroup on $h^{\alpha}(\Tone, \mathbb{C})$ 
with domain $h^{4 + \alpha}(\Tone, \mathbb{C})$. Since $h^{\alpha}_0(\Tone, \mathbb{C})$ inherits the topology 
of $h^{\alpha}(\Tone, \mathbb{C})$ and the projection $P_0$ commutes with 
$DG_{\star}(\eta)$, the analogous resolvent estimates hold for 
$D_1 \mathcal{G}_{\star}(0,\eta)$ and so we see that
$D_1 \mathcal{G}_{\star}(0,\eta)$ generates an analytic semigroup on $h^{\alpha}_0(\Tone, \mathbb{C})$ 
with domain $h^{4 + \alpha}_0(\Tone, \mathbb{C})$.
Moreover, from \eqref{Eqn:Spectra} it holds that $type(D_1 \mathcal{G}_{\star}(0,\eta)) < 0$
for all $\eta \in (- \delta, \delta)$, where $type(B)$ denotes the 
spectral type of the semigroup generator $B$. In particular, we have 
\[
type(D_1 \mathcal{G}_{\star}(0, \eta)) < \frac{1 - (\rs - \delta)^{2}}{(\rs - \delta)^2} < 0 ,
\qquad \eta \in (- \delta, \delta).
\]
Now, choose $\omega \in \Big( 0, \frac{(\rs - \delta)^{2} - 1}{(\rs - \delta)^2} \Big)$ 
and the remainder of the result follows from \cite[Theorem III.3.4.1 and Remarks 3.4.2(b)]{AM95} 
and the restriction of maximal regularity 
from the complex--valued spaces $h^{\sigma}_0(\Tone, \mathbb{C})$ 
to the subspaces $h^{\sigma}_0(\Tone)$. 
\end{proof}


\subsection{Exponential Stability of Cylinders with Radius $\rs > 1$}

Our main result of this section
establishes exponential asymptotic stability 
of the family of cylinders,
by which we mean that small perturbations of a cylinder $\Gamma(\rs)$ 
will have global solutions which converge exponentially fast to a cylinder $\Gamma(\rs + \eta)$, 
where $\rs \not= \rs + \eta$ in general.
Before formulating our result, we recall that
$\mathbb{B}_E(a,\varepsilon)$ denotes the open ball with center $a$ and radius $\varepsilon$,
in the normed vector space $E$. 
In particular, $\mathbb{B}_{h^{2 + \alpha}}(r, \varepsilon)$
consists of all functions in $h^{2 + \alpha}(\Tone)$ which are 
close to $r$ in the $C^{2 + \alpha}$ topology.

\begin{thm}[Exponential Stability]\label{Thm:Stability}
Fix $\alpha \in (0,1)$, $\mu \in [1/2, 1],$ so that $4 \mu + \alpha \notin \mathbb{Z}$, 
and $\rs > 1$. There exist nonzero positive constants $\varepsilon = \varepsilon(\rs)$,
$\delta = \delta(\rs)$ and $\omega = \omega(\rs, \delta),$ such that problem \eqref{ASD2} with
initial data $r_0 \in \mathbb{B}_{h^{4 \mu + \alpha}} (\rs, \varepsilon)$ has a unique global solution
\[
r(\cdot, r_0) \in C_{1 - \mu}^1(\mathbb{R}_+, h^{\alpha}(\Tone)) \cap C_{1 - \mu}(\mathbb{R}_+, h^{4 + \alpha}(\Tone)),
\]
and there exists $\eta = \eta(r_0) \in (-\delta, \delta)$ and $M = M(\alpha) > 0$ for which the bound
\[
t^{1 - \mu} \|r(t, r_0) - (\rs + \eta) \|_{h^{4 + \alpha}} + \|r(t, r_0) - (\rs + \eta) \|_{h^{4 \mu + \alpha}} 
\leq  e^{-\omega t} M \|r_0 - \rs \|_{h^{4 \mu + \alpha}}
\] 
holds uniformly for $t > 0$.
\end{thm}

\begin{proof}
{\bf (i)} Let $\delta, \; \omega > 0$ be the constants given by Theorem~\ref{Thm:ProjMaxReg} 
and consider the operator
\[
\mathcal{K}(\trho, \trho_0, \eta) := \Big( \partial_t \trho - \mathcal{G}_{\star}(\trho, \eta), \; \gamma \trho - \trho_0  \Big),
\]
acting on $\mathbb{U} := \Big( \mathbb{F}_1(\mathbb{R}_+, \omega) \cap C(\mathbb{R}_+, U_0) \Big) \times \Big(U_0 \cap F_{\mu} \Big) \times U_1$
which is open in the Banach space $\mathbb{F}_1(\mathbb{R}_+, \omega) \times F_{\mu} \times \mathbb{R}$. 

First, we show that $\mathcal{K}$ maps $\mathbb{U}$ into $\mathbb{F}_0(\mathbb{R}_+, \omega) \times F_{\mu}$.
Notice that
\[
\gamma: \mathbb{F}_1(\mathbb{R}_+, \omega) \rightarrow (F_0, F_1)_{\mu, \infty}^0
\]
follows from \cite[Lemma 2.2(a)]{CS01}, so $\gamma \trho \in F_{\mu}$.
Meanwhile, $\partial_t$ maps $\mathbb{F}_1(\mathbb{R}_+, \omega)$ into $\mathbb{F}_0(\mathbb{R}_+, \omega)$
by definition of the spaces $BU\!C_{1 - \mu}^1(J, E)$. Finally, to see that $\mathcal{G}_{\star}(\cdot, \eta)$ maps 
$\mathbb{U}$ into 
$\mathbb{F}_0(\mathbb{R}_+, \omega)$, choose $\trho \in \mathbb{U}$ and
notice that $\trho(t) \in U_0 \cap h^{2 + \alpha}_0(\Tone)$, for $t > 0$, from the embeddings \eqref{Eqn:Embedding}.
Utilizing the explicit quasilinear representation of the operator $G$, as given by
\eqref{Eqn:ADefined}--\eqref{Eqn:FDefined}, whereby
\[
\mathcal{G}_{\star}(\trho(t), \eta) = 
P_0 \Big( - \mathcal{A}\Big( \psi(\trho(t), \eta) + \rs \Big) (\psi(\trho(t),\eta) + \rs) + f \big(\psi(\trho(t), \eta) + \rs \big) \Big),
\]
one will easily conclude the desired mapping property for the operator $\mathcal{G}_{\star}$. For instance,
we have seen that $\mathcal{A}(\rho)\rho = b_1(\rho) \partial_x^4 \rho + b_2(\rho) \partial_x^3 \rho$, where the
functions $b_i$ only depend on $\rho, \rho_x$ and $\rho_{xx}$, $i = 1,2$. Hence, it follows that
\begin{align*}
e^{\omega t} t^{1 - \mu} &\left\| \mathcal{A} \Big( \psi(\trho(t), \eta) + \rs \Big) (\psi(\trho(t), \eta) + \rs) \right\|_{E_0}\\
&\leq  e^{\omega t} t^{1 - \mu} \left\| \partial_x^4 \psi(\trho(t), \eta) \right\|_{E_0}
\left\| b_1 \big( \psi(\trho(t), \eta) + \rs \big) \right\|_{E_0} \\
&\quad + \, e^{\omega t} t^{1 - \mu} \left\| \partial_x^3 \psi(\trho(t), \eta) \right\|_{E_0}
\left\| b_2 \big( \psi(\trho(t), \eta) + \rs \big) \right\|_{E_0},
\end{align*}
for $t > 0$.
From here, we take advantage of the boundedness of $\psi(\trho(t), \eta)$ in the topology of $F_{1/2}$,
in conjunction with the explicit formulas for $b_i$, in order to bound the terms $\| b_i (\psi(\trho(t), \eta) + \rs) \|_{E_0}$,
uniformly in $t$. Meanwhile, the representation given by Remarks~\ref{Rem:PsiProperties}(d) and
the fact that $\trho \in \mathbb{F}_1(\mathbb{R}_+, \omega)$ yield the bounds 
\[
e^{\omega t} t^{1 - \mu} \| \partial_x^k \psi(\trho(t), \eta) \|_{E_0} = e^{\omega t} t^{1 - \mu} \| \partial_x^k \trho(t) \|_{F_0} \le \| e^{\omega t} \trho \|_{\mathbb{F}_1(\mathbb{R}_+)},
\qquad k = 1, \ldots, 4.
\] 
Analogous methods work for the remaining terms of the function $G_{\star} \big( \psi(\trho(t), \eta) \big),$
since we can always isolate an element of the form $\partial_x^k \psi(\trho(t), \eta),$  and bound the remaining
elements using boundedness in $F_{1/2}$.
We conclude the result by noting that the linear projection $P_0$ adds no complexity to acquiring the 
necessary bounds.
 
Regarding the regularity of $\mathcal{K}$, it can be
shown that $\mathcal{G}_{\star}$ is $C^{\omega}$ via substitution operators and 
the derivative $\partial_t$
and the trace operator $\gamma$ are linear. 
Hence, it follows that
\[
\mathcal{K} \in C^{\omega} \Big( \mathbb{U}, 
\mathbb{F}_0(\mathbb{R}_+, \omega) \times F_{\mu} \Big). \\
\]
Meanwhile, notice that $\mathcal{K}(0,0,0) = (0, 0)$ and 
\[ 
D_1 \mathcal{K}(0,0,0) = \Big( \partial_t - D_1 \mathcal{G}_{\star}(0,0), \gamma  \Big) \in \mathcal{L}_{isom}
\Big(\mathbb{F}_1(\mathbb{R}_+, \omega), \mathbb{F}_0(\mathbb{R}_+, \omega) 
\times F_{\mu} \Big),
\]
by Theorem~\ref{Thm:ProjMaxReg}. Hence, we conclude from the implicit function 
theorem that there exists an open neighborhood 
$0 \in \tilde{U} \subset F_{\mu} \times \mathbb{R}$ and a $C^{\omega}$ mapping 
$\kappa : \tilde{U} \rightarrow \mathbb{F}_1(\mathbb{R}_+, \omega)$ such that 
\[
\mathcal{K}(\kappa(\trho_0, \eta), \trho_0, \eta) = (0, 0) \quad \text{for all} \quad (\trho_0, \eta) \in \tilde{U}.
\]
In particular, $\kappa(\trho_0, \eta)$ is a global solution to \eqref{PASD} with parameter $\eta$ and 
initial data $\trho_0 \in F_{\mu}$, where we assume, without loss of generality, that
$\tilde{U} \subseteq U$.

\medskip

{\bf (ii)} Choose $\varepsilon > 0$ so that for every $r_0 \in \mathbb{B}_{E_{\mu}}(\rs, \varepsilon)$,
there exists $\eta \in (-\rs, \infty)$ for which 
\[
(P_0r_0, \eta) \in \tilde{U} \quad \text{and} \quad F_{\star}(r_0 - \rs; \rs) = F_{\star}(\eta; \rs).
\]
The existence of such a constant $\varepsilon$ is guaranteed by the continuity of $P_0$ and $F_{\star}$,
injectivity of $F_{\star}(\eta; \rs)$ for $\eta \in (- \rs, \infty)$ and the fact that $P_0 \rs = 0$.

Let $r_0 \in \mathbb{B}_{E_{\mu}}(\rs, \varepsilon)$ and fix $\eta = \eta(r_0)$ as mentioned
so that $F_{\star}(r_0 - \rs) = F_{\star}(\eta).$ Define the function
\begin{equation}\label{Eqn:Solution}
r := \psi( \kappa(P_0 r_0, \eta), \eta) + \rs,
\end{equation}
where $\psi(\kappa(P_0 r_0, \eta), \eta)(t) := \psi(\kappa(P_0 r_0, \eta)(t), \eta)$, and 
we will demonstrate that $r$ satisfies the desired properties claimed in the theorem.

To see that $r$ is the unique global solution to \eqref{ASD2} with initial data $r_0$, first
fix $T > 0$ and consider the interval $J := [0, T]$. By the choice of $\varepsilon > 0$ we know that 
$(P_0r_0, \eta) \in \tilde{U}$ and so it follows from part {\bf (i)} above that 
$\kappa(P_0r_0, \eta) \in \mathbb{F}_1(\mathbb{R}_+, \omega)$. From this we see that 
$\kappa(P_0r_0, \eta) \in \mathbb{F}_1(J)$ is a solution to \eqref{PASD} with initial data
$P_0r_0 \in F_{\mu}$. Thus it follows, by Lemma~\ref{Lem:PsiLifts},
that $r \in \mathbb{E}_1(J)$ is the solution on $J$ to the problem \eqref{ASD2} with
initial data 
\[
\psi(P_0r_0, \eta) + \rs = \psi(P_0(r_0 - \rs), \eta) + \rs = r_0,
\]
where we use Remarks~\ref{Rem:PsiProperties}(b) and the fact that $r_0 - \rs \in \mathcal{M}_{\eta}^{4 \mu + \alpha}$.
The claim now follows by the fact that $T > 0$ was arbitrary and by definition of the 
Fr\'echet spaces $C_{1 - \mu}(\mathbb{R}_+, E)$.

To see that $r$ satisfies the exponential bounds in the second part of the claim,
first notice that $\kappa(0, \eta) \equiv 0$ for $\eta \in U_1$. Then,
by Remarks~\ref{Rem:PsiProperties}, and application of the mean value theorem, 
the expression
\begin{align*}
r(t) - &(\rs + \eta) = \psi(\kappa(P_0r_0, \eta)(t), \eta) - \eta = \psi(\kappa(P_0r_0, \eta)(t), \eta) - \psi(\kappa(0, \eta)(t), \eta)\\
&= \Big( P_0 + (1 - P_0) \Big) \Big( \psi(\kappa(P_0r_0, \eta)(t), \eta) - \psi(\kappa(0, \eta)(t), \eta) \Big)\\
&= \kappa(P_0r_0, \eta)(t) + \frac{1}{2 \pi} \int_{\Tone} \Big( \psi(\kappa(P_0r_0, \eta)(t, x), \eta) - \psi(\kappa(0, \eta)(t, x), \eta) \Big) dx\\
&= \kappa(P_0r_0, \eta)(t) + \frac{1}{2 \pi} \int_{\Tone} \int_0^1 D_1 \psi \big( \tau \kappa(P_0r_0, \eta)(t), \eta \big) \kappa(P_0r_0, \eta)(t,x) d\tau dx,
\end{align*}
holds for all $t > 0$.
Notice that
\[
e^{\omega t} t^{1 - \mu} \| \kappa(P_0r_0, \eta)(t) \|_{F_1} \leq \| \kappa(P_0r_0, \eta) \|_{\mathbb{F}_1(\mathbb{R}_+, \omega)}\\
\]
and
\[
\sup_{t \in \mathbb{R}_+} \| e^{\omega t} \kappa(P_0r_0, \eta)(t) \|_{F_{\mu}}
\]
are finite quantities by the fact that $\kappa(P_0r_0, \eta) \in \mathbb{F}_1(\mathbb{R}_+, \omega)$ and the embedding 
\eqref{Eqn:Embedding}. We note that the reference for \eqref{Eqn:Embedding} does not explicitly include the 
unbounded interval $J = \mathbb{R}_+$, however the methods of the proof extend to this 
unbounded case with little trouble. Meanwhile, the remaining term in $r(t) - (\rs + \eta)$ above
is scalar-valued, so we bound $D_1 \psi(\tau \kappa(P_0r_0, \eta)(t), \eta) \kappa(P_0r_0, \eta)(t)$
in the $C(\Tone)$-topology, which is then bounded in the $h^{\sigma}(\Tone)$-topology, for any 
$\sigma \in \mathbb{R}_+ \setminus \mathbb{Z}$. In particular, observe that, by \eqref{Eqn:PsiBound},
\begin{equation*}
\sup_{\trho \in U_0} \| D_1 \psi(\trho, \eta) \kappa(P_0r_0, \eta)(t) \|_{h^{\alpha}} \leq N \| \kappa(P_0r_0, \eta)(t) \|_{h^{\alpha}_0}, \qquad t > 0,
\end{equation*}
and we conclude that the bounds
\begin{equation}\label{Eqn:Bound2}
e^{\omega t} t^{1 - \mu} \| r(t) - (\rs + \eta) \|_{E_1} \leq \Big(1 + c_1 N \Big) \| \kappa(P_0r_0, \eta) \|_{\mathbb{F}_1(\mathbb{R}_+, \omega)}
\end{equation}
and 
\begin{equation}\label{Eqn:Bound3}
e^{\omega t} \| r(t) - (\rs + \eta) \|_{E_{\mu}} \leq \Big( c_2 + c_3 N \Big) \| \kappa(P_0r_0, \eta) \|_{\mathbb{F}_1(\mathbb{R}_+, \omega)},
\end{equation}
hold uniformly for $t > 0$. Here the constant $c_1$ comes from the embedding 
$F_1 \hookrightarrow F_0$, and the constants $c_2$ and $c_3$ come from 
the embeddings \eqref{Eqn:Embedding}. Finally, by the regularity of 
$\kappa$, we may assume that $\tilde{U}$ was chosen sufficiently small to ensure that $D_1 \kappa$ is uniformly
bounded from $\tilde{U}$ into $\mathbb{F}_1(\mathbb{R}_+, \omega).$
Recalling that $\kappa(0, \eta) = 0$, it follows that
\begin{equation}\label{Eqn:Kappa}
\begin{split}
\| \kappa(P_0 r_0, \eta) \|_{\mathbb{F}_1(\mathbb{R}_+, \omega)} &\leq \int_0^1 \left\| D_1 \kappa(\tau P_0 r_0, \eta) P_0 r_0 \right\|_{\mathbb{F}_1(\mathbb{R}_+, \omega)} \; d \tau\\
&\leq \tilde{M} \| P_0 r_0 \|_{F_{\mu}} \leq M \|r_0 - \rs\|_{E_{\mu}},
\end{split}
\end{equation}
where $M := \| P_0 \| \sup_{(\trho, \eta) \in \tilde{U}} \|D_1 \kappa(\trho, \eta) \|_{\mathcal{L}(
F_{\mu}, \mathbb{F}_1(\mathbb{R}_+, \omega))}.$
The claim now follows from \eqref{Eqn:Kappa} and
the inequalities \eqref{Eqn:Bound2}--\eqref{Eqn:Bound3}.
\end{proof}

\begin{remark}
With Theorem~\ref{Thm:Stability} established, we note that the equivolume manifold
$\mathcal{M}^{2 + \alpha} = \mathcal{M}^{2 + \alpha}(\rs) \subset h^{2 + \alpha}(\Tone)$
is a local stable manifold for the cylinder of radius $\rs > 1$. Moreover, these
manifolds foliate the interval $(1, \infty) \subset h^{2 + \alpha}(\Tone)$
with the radius $\rs$ a parameter which separates leaves of the foliation.
\end{remark}

\section{Instability of Cylinders with Radius $0 < \rs < 1$}\label{Sec:Instability}

Taking advantage of the reduced problem \eqref{PASD} and the
connection with
\eqref{ASD},
we proceed with the following result regarding instability of cylinders with
radius $0 < \rs < 1$, in the setting of $F_{\mu}$.
Differences in volume
between the initial data $r_0$ and the cylinder $\rs$ are not a factor
in the following argument, so we assume that the 
parameter $\eta$, associated with the reduced problem \eqref{PASD}, is
simply taken to be zero for this proof.

\begin{thm}\label{Thm:Instability}
Let $\rs \in (0, 1)$ and $\mu \in [1/2,1]$ such that $4 \mu + \alpha \notin \mathbb{Z}$.
Then the equilibrium $\rs$ of \eqref{ASD} is unstable in the topology of $h^{4 \mu + \alpha}(\Tone)$
for initial values in $h^{4 \mu + \alpha}(\Tone)$. More precisely, 
there exists $\varepsilon > 0$ and a sequence of initial values
$(r_n) \subset h^{4 \mu + \alpha}(\Tone)$ such that
\begin{itemize}
	\item[$\bullet$] $\lim_{n \rightarrow \infty} \| r_n - \rs \|_{h^{4 \mu + \alpha}} = 0$, \qquad and 
	\item[$\bullet$] for each $n$, there exists $t_n \in J(r_n)$ so that 
		$\| r(t_n, r_n) - \rs \|_{h^{4 \mu + \alpha}} \ge \varepsilon$.	
\end{itemize}
\end{thm}

\begin{proof}
{\bf (i)}
We begin by showing that $0$ is an unstable equilibrium for the reduced problem \eqref{PASD} centered at $\rs$.
Let $L := D_1 \mathcal{G}_{\star}(0, 0)$ be the linearization of 
$\mathcal{G}_{\star}$ at $\trho = 0$.
We can restate the evolution equation \eqref{PASD} 
in the equivalent form
\begin{equation}
\label{reduced-2}
\begin{cases}
\trho_t - L \trho = g(\trho), &\text{$t > 0$} \\ \trho(0)=\trho_0,
\end{cases}
\end{equation}
where $g(\trho) := \mathcal{G}_{\star}(\trho, 0) - L \trho$.
Using the quasilinear structure of  $[\trho \mapsto \mathcal{G}_{\star}(\trho, 0)]$ 
it is not difficult to see that for every $\beta > 0$ there exists a 
number $\varepsilon_0 = \varepsilon_0(\beta) > 0$ such that
\begin{equation}
\label{estimate-g}
\|g(\trho)\|_{F_0} \leq \beta \| \trho \|_{F_1}, \quad 
\trho \in \mathbb{B}_{F_{\mu}}(0, \varepsilon_0) \cap F_1,
\end{equation}
where we will be assuming throughout that $\trho \in U_0$, 
to guarantee that $\mathcal{G}_{\star}(\trho, 0)$, and subsequently $g(\trho)$,
is defined.
It follows from \eqref{Eqn:Spectra} that
\begin{equation*}
\sigma(L) \cap [{\rm Re} \, z > 0] \ne \emptyset,
\end{equation*}
and we may choose numbers $\omega, \gamma > 0$ such that
\begin{equation*}
[\omega - \gamma \le {\rm Re} \,z \le \omega + \gamma] \cap \sigma(L) = \emptyset
\quad \text{and} \quad
\sigma_+ := [{\rm Re} \, z > \omega + \gamma] \cap \sigma(L) \ne \emptyset \, ,
\end{equation*}
i.e. the strip $[\omega - \gamma \le {\rm Re} \, z \le \omega + \gamma]$ does not intersect
$\sigma(L)$ and there is at least one point of $\sigma(L)$ to the right  of
the line $[{\rm Re} \, z = \omega + \gamma]$.

We define $P_+$ to be the spectral projection, in $F_0$, 
with respect to the spectral set $\sigma_+$,
and let $P_- := 1 - P_+$. 
Then $P_+(F_0)$ is finite dimensional 
and the topological decomposition  
\[
F_0 = P_+(F_0) \oplus P_-(F_0)
\]
reduces $L$, so that $L = L_+ \oplus L_-$,
where $L_\pm$ is the part of $L$ in $P_{\pm}(F_0)$, respectively,
with the domains $D(L_\pm) = P_\pm(F_1)$. 
Moreover, $P_{\pm}$ decomposes $F_1$ by the embedding $F_1 \hookrightarrow F_0$,  
and, without loss of generality, we can take the norm on $F_1$
so that
\[
\| v \|_{F_1} = \| P_+ v \|_{F_1} + \| P_- v \|_{F_1}.
\]
We note that
\begin{equation*}
\sigma(L_-) \subset [{\rm Re} \, z < \omega - \gamma],
\qquad 
\sigma(L_+) = \sigma^+ \subset [{\rm Re} \, z > \omega + \gamma].
\end{equation*}
This implies that there is a constant $M_0 \ge 1$ such that
\begin{equation}
\label{estimate-sg}
\begin{split}
\|e^{L_{-} t} P_{-} \|_{\mathcal{L}(F_0)} &\le M_0 e^{(\omega - \gamma) t}, \\
\|e^{-L_{+}t} P_{+} \|_{\mathcal{L}(F_0)} &\le M_0 e^{-(\omega + \gamma) t}, \qquad t \ge 0,
\end{split}
\end{equation}
where $\{ e^{L_{-}t} : t \ge 0 \}$ is the analytic semigroup in $P_-(F_0)$ generated by $L_{-}$ and
$\{ e^{L_{+}t} : t \in \mathbb{R} \}$ is the group in $P_+(F_0)$
generated by the bounded operator $L_+$.

From \eqref{Eqn:D1PG}--\eqref{Eqn:DG} and \cite[Theorem 5.2]{LeC11} one sees that 
$\big(\mathbb{F}_0(J), \mathbb{F}_1(J) \big)$
is a pair of maximal regularity for $-L$ and 
it is easy to see that $-L_-$ inherits the property of maximal regularity.
In particular, the pair $\big( P_- (\mathbb{F}_0(J)), P_-(\mathbb{F}_1(J)) \big)$
is a pair of maximal regularity for $-L_-$.
In fact, since $type( - \omega + L_{-} ) < - \gamma < 0$ 
we see that $\big( P_-(\mathbb{F}_0(\mathbb{R}_+)), P_-(\mathbb{F}_1(\mathbb{R}_+)) \big)$
is a pair of maximal regularity for $(\omega - L_-)$.
This, in turn, implies the a priori estimate
\begin{equation}
\label{estimate-stable}
\| e^{ - \omega t} w \|_{\mathbb{F}_1(J_T)}
\le M_1 \Big( \| w_0 \|_{F_{\mu}} + \| e^{- \omega t} f \|_{\mathbb{F}_0(J_T)}\Big)
\end{equation}
for $J_T := [0,T]$, any $T \in (0, \infty)$
(or $J_T = \mathbb{R}_+$ for $T = \infty$), 
with a universal constant $M_1 > 0$, where
$w$ is a solution of the linear Cauchy problem
\begin{equation*}
\begin{cases} \dot{w} - L_- w = f,\\ w(0) = w_0, \end{cases}
\end{equation*}
with $(f, w_0) \in \Big( C\big( (0,T), P_- F_0 \big), P_- U_{0} \Big).$ 

\medskip

{\bf (ii)} By way of contradiction, suppose that the equilibrium $0$ is stable for \eqref{PASD}.
Then for every $\varepsilon > 0$ there exists a number $\delta > 0$ such that 
\eqref{reduced-2} admits
for each $\trho_0 \in \mathbb{B}_{F_{\mu}}(0, \delta)$ 
a global solution
\[
\trho = \trho(\cdot, \trho_0) \in C^1_{1 - \mu}(\mathbb{R}_+, F_0) \cap C_{1 - \mu}(\mathbb{R}_+, F_1) \cap C(\mathbb{R}_+, U_0),
\]
which satisfies 
\begin{equation}
\label{less-eps}
\| \trho(t) \|_{F_{\mu}} < \varepsilon, \qquad t \ge 0.
\end{equation}
We can assume without loss of generality that $\beta$ and $\varepsilon$ are chosen such that
\begin{equation}
\label{beta}
2 C_0(M_0 + M_1 \gamma) \beta \le \gamma \qquad \text{and} \qquad \varepsilon \le \varepsilon_0(\beta),
\end{equation}
where $C_0 := \max \{\| P_- \|_{\mathcal{L}(F_0)}, \| P_+ \|_{\mathcal{L}(F_0)} \}$.
As $P_+(F_0)$ is finite dimensional, we may also assume that 
\begin{equation*}
\|P_+ v \|_{F_{\nu}} = \| P_+ v \|_{F_0}, \qquad v \in F_0, \quad \nu \in \{ \mu, 1 \},
\end{equation*} 
where we also use the fact that $P_+ F_0 \subset D(L^n)$ for every $n \in \mathbb{N}$,
c.f. \cite[Proposition A.1.2]{LUN95}.

\medskip

{\bf CLAIM 1:} For any initial value
$\trho_0 \in \mathbb{B}_{F_{\mu}}(0, \delta)$,
$P_+ \trho$ admits the representation
\begin{equation}
\label{P-plus-formula}
P_+ \trho(t) = - \int_t^{\infty} e^{L_+(t - s)} P_+ g( \trho(s) ) \, ds \qquad t \ge 0.
\end{equation}
For this we first establish that, for any $\trho_0 \in \mathbb{B}_{F_{\mu}}(0, \delta)$,
\[
e^{- \omega t} \trho \in BC_{1 - \mu}(\mathbb{R}_+, F_1) := \left\{ u \in C((0, \infty), F_1): \sup_{t \in \mathbb{R}_+} t^{1 - \mu} \| u(t) \|_{F_1} < \infty \right\}.
\]
First notice that the mapping property 
\[
g : \mathbb{F}_1(J_T) \cap C(J_T, U_0) \rightarrow \mathbb{F}_0(J_T), \qquad 0 < T < \infty,
\] 
which follows in the same way as the mapping property derived for $\mathcal{G}_{\star}$
in the proof of Theorem~\ref{Thm:Stability} above, together
with the inequalities \eqref{estimate-g} and \eqref{estimate-stable} yield
\begin{equation}
\label{P-minus-1}
\begin{split}
&\| e^{-\omega t} P_{-} \trho \|_{B_{1 - \mu}(J_T, F_1)}\\
&\le  M_1 \Big( \|P_{-} \trho_0 \|_{F_{\mu}}
+ C_0 \beta \| e^{-\omega t} P_+ \trho \|_{B_{1 - \mu}(J_T, F_1)}
+ C_0 \beta \| e^{-\omega t} P_- \trho \|_{B_{1 - \mu}(J_T, F_1)}
\Big)
\end{split}
\end{equation}
for any $0 < T < \infty$.
Due to \eqref{beta}, we have $M_1 C_0 \beta \le 1/2$ and can further conclude
\begin{equation}
\label{estimate-Pminus-2}
\begin{split}
\| e^{-\omega t} P_{-} \trho \|_{B_{1 - \mu}(J_T, F_1)}
\le  2 M_1 \Big( \| P_{-} \trho_0 \|_{F_{\mu}}
+ C_0 \beta \| e^{-\omega t} P_+ \trho \|_{B_{1 - \mu}(J_T, F_1)}
\Big).
\end{split}
\end{equation}
It follows from \eqref{less-eps} that
\begin{equation*}
t^{1 - \mu} \| e^{-\omega t} P_+ \trho(t) \|_{F_1}
\le t^{1 - \mu} e^{-\omega t} C_0 \| \trho(t) \|_{F_{\mu}} \le C_0 C_1 \varepsilon
\end{equation*}
where $C_1 := \sup \{ t^{1 - \mu} e^{-\omega t} : t \ge 0 \} < \infty$.
Inserting this result into \eqref{estimate-Pminus-2} yields
\begin{equation}
\label{estimate-rho-1}
\| e^{-\omega t} \trho \|_{B_{1 - \mu}(J_T, F_1)}
\le 2 M_1 \| P_{-} \trho_0 \|_{F_{\mu}} + (2 M_1 C_0 \beta + 1) C_0 C_1 \varepsilon \le C_2
\end{equation}
for any $0 < T < \infty$.
However, since $T$ is arbitrary and 
\eqref{estimate-rho-1} is independent of $T$ we conclude that 
$e^{-\omega t} \trho \in BC_{1 - \mu}(\mathbb{R}_+, F_1)$, for any initial value 
$\trho_0 \in \mathbb{B}_{F_{\mu}} (0, \delta)$.
Next we note that, for $s \ge t$, by \eqref{estimate-sg}
\begin{equation}
\label{estimate-integral-1}
\begin{split}
\| e^{L_+(t-s)} P_+ g(\trho(s)) \|_{F_0}
&\le M_0 C_0 \beta e^{(\omega + \gamma)(t-s)} \| \trho(s) \|_{F_1}\\
&\le M_0 C_0 \beta e^{\omega t} e^{\gamma (t-s)} s^{\mu - 1} \| e^{-\omega s} \trho \|_{B_{1 - \mu}(\mathbb{R}_+, F_1)},
\end{split}
\end{equation}
which shows that the integral in \eqref{P-plus-formula} exists for any $t \ge 0$,
with convergence in $F_1$.
Moreover, 
\begin{equation}
\label{estimate-integral-2}
\begin{split}
\left\| \int_t^{\infty} e^{L_+ (t-s)} P_+ g(\trho(s)) \, ds \right\|_{F_0}
\le e^{\omega t} M_0 C_0 C_3 \beta \| e^{-\omega t} \trho \|_{B_{1 - \mu}(\mathbb{R}_+, F_1)},
\end{split}
\end{equation}
where $C_3 := \sup \big\{ \int_t^{\infty} e^{\gamma (t-s)} s^{\mu - 1} \, ds : t \ge 0 \big\} < \infty$.
Noting that $w = P_+ \trho$ solves the Cauchy problem
\begin{equation*}
\begin{cases} \dot{w} - L_+ w = P_+ g(\trho),\\ w(0) = P_+ \trho_0, \end{cases}
\end{equation*}
it follows from the variation of parameters formula that, for $t \ge 0$ and $\tau > 0$,
\begin{equation*}
\label{plus-unstable}
P_+ \trho(t) =
e^{L_+ (t - \tau)} P_+ \trho(\tau) + \int_{\tau}^t e^{ L_+(t - s)} P_+ g(\trho(s)) \, ds.
\end{equation*}
Since this representation holds for any $\tau > 0$, the claim 
follows from \eqref{estimate-sg} and \eqref{less-eps} by sending $\tau$ to $\infty$.

\medskip

{\bf CLAIM 2:} For any $\trho_0 \in \mathbb{B}_{F_{\mu}}(0, \delta)$
it must hold that
\[
\| P_+ \trho_0 \|_{F_{\mu}} \leq 2 M_0 M_1 C_3 \| P_- \trho_0 \|_{F_{\mu}}.
\]
From \eqref{P-plus-formula} and \eqref{estimate-integral-1} follows
\begin{equation}
\label{P-plus-2}
\begin{split}
&\| e^{-\omega t} P_+ \trho \|_{B_{1 - \mu}(\mathbb{R}_+, F_0)}\\
&\le \frac{M_0 C_0 \beta}{\gamma} 
\Big( \| e^{-\omega t} P_+ \trho \|_{B_{1 - \mu}(\mathbb{R}_+, F_1)}
+ \| e^{-\omega t} P_- \trho \|_{B_{1 - \mu}(\mathbb{R}_+, F_1)} \Big)
\end{split}
\end{equation}
where we have used the fact that
$\sup_{t \ge 0} \big\{ t^{1 - \mu} \int_t^{\infty} e^{\gamma (t-s)} s^{\mu - 1} \, ds \big\} \le 1/\gamma$.
Adding the estimates in \eqref{P-minus-1} and \eqref{P-plus-2} and 
employing \eqref{beta} yields
\begin{equation}
\label{B}
\|e^{-\omega t} \trho \|_{B_{1 - \mu}(\mathbb{R}_+, F_1)} \le 2 M_1 \| P_- \trho_0 \|_{F_{\mu}}.
\end{equation}
The representation \eqref{P-plus-formula} in 
conjunction with \eqref{estimate-integral-2} and \eqref{B} then implies
\begin{equation}
\begin{split}
\|P_+ \trho_0 \|_{F_{\mu}}
\le M_0 C_0 C_3 \beta \| e^{-\omega t} \trho \|_{B_{1 - \mu}(\mathbb{R}_+, F_1)}
 \le M_0 C_3 \|P_- \trho_0 \|_{F_\mu},
\end{split}
\end{equation}
where the last inequality follows from the fact that $2 C_0 M_1 \beta \le 1$.
We have thus demonstrated the claim. 

\medskip

Notice that the preceding claim 
contradicts the stability assumption. In particular,
if $\trho_0 \in \mathbb{B}_{F_{\mu}}(0, \delta)$, $\trho_0\neq 0$, is chosen such that
$P_- \trho_0 = 0$, then it must hold that $P_+ \trho_0 = 0$, and hence $\trho_0 = 0$,
leading to a clear contradiction.
In particular, we conclude that there exists $\tilde{\varepsilon} > 0$ and a sequence 
$(\trho_n) \subset F_{\mu}$ such that $\trho_n \rightarrow 0$ and the solution $\trho(\cdot, \trho_0)$
satisfies $\| \trho(t_n, \trho_n) \|_{F_\mu} \ge \tilde{\varepsilon}$
for some $t_n \in J(\trho_n)$.

\medskip

{\bf (iii)} The result now follows by application of the projection $P_0$ to perturbations of 
the cylinder $\rs$. In particular, stability of the $\rs$ for \eqref{ASD} would necessarily
imply stability of 0 for \eqref{PASD}, which contradicts the conclusion of parts {\bf (i)} and {\bf(ii)}.
\end{proof}

\begin{remark}
With explicit knowledge of eigenvalues for the linearization, we can also conclude existence of stable,
unstable, and center manifolds for \eqref{ASD} about cylinders $\rs \in (0,1)$.
See \cite{DPL88,Sim95} for existence of such manifolds for nonlinear parabolic problems with
continuous maximal regularity. The characterization of the unstable manifolds is a very interesting
open question, especially regarding investigation of pinch--off behavior.
\end{remark}

\section{Bifurcation Results}\label{Sec:Bifurcation}

In this section we turn our attention 
to interactions between the family of cylinders and the family of unduloids. We have already seen
that the radius $\rs = 1$ plays a critical role in the dynamics of the cylinders. The change of stability 
for cylinders above and below this critical radius suggests that there is a bifurcation at $\rs = 1$. Indeed, 
we will confirm this bifurcation, using results of Crandall and Rabinowitz \cite{CR71}, and investigate  
properties of the bifurcation. Herein we take the parameter $\lambda := 1 / \rs$ as our 
bifurcation parameter, $\rs > 0$. 

From the reductions developed in Section~\ref{Sec:Stability}, it suffices to study 
the bifurcation equation
\begin{equation}\label{Eqn:Bifurcation}
\bG(\trho, \lambda) := \mathcal{G}_{\star}(\trho, 0) = P_0 G( \psi( \trho) + \rs) = 0, \qquad \lambda = 1 / \rs,
\end{equation}
in the setting of $(\trho, \lambda) \in F_1 \times (0, \infty)$, where we use $\psi(\trho) := \psi(\trho, 0)$
to economize notation. Recalling the explicit characterization \eqref{Eqn:Spectra}, 
we note that the eigenvalues of $D_1 \mathcal{G}_{\star}(0, 0)$ all have multiplicity two in the setting of 
$F_1$, regardless of the value of the parameter $\rs$. 
The techniques of 
\cite{CR71}, where the authors derive results for operators with simple 
eigenvalues, are not directly applicable in this setting. We may choose at this point to employ more 
general bifurcation results for high dimensional kernels, such as the results contained in 
\cite[Section I.19]{KIE12}, or we can simplify our setting 
to make accessible the results of \cite{CR71}. 

Whether we choose to simplify our setting or use higher dimensional bifurcation
results, we can make good use of the following observation. Due to the periodicity enforced
in the problem, the set of equilibria of \eqref{ASD} is invariant under shifts along the axis of 
rotation. More precisely, considering the translation operators $T_a,$ discussed in the proof of 
Proposition~\ref{Prop:SolutionRegularity} in the appendix, one can easily verify that $G(T_a \bar{r}) = 0$ if and 
only if $G(\bar{r}) = 0$, $a \in \mathbb{R}$. Obviously, this invariance carries over to the 
reduced problem \eqref{PASD} and subsequently to the bifurcation equation \eqref{Eqn:Bifurcation}.

One can take advantage of this shift invariance of equilibria in the context of bifurcation 
with high dimensional kernels
by constructing a two dimensional bifurcation parameter $\tilde{\lambda} = (1/\rs, a)$ and eventually observes 
two dimensional bifurcating \emph{surfaces} of equilibria in $F_1$, 
c.f. \cite[Theorem I.19.2 and Remark I.19.7]{KIE12}. 
On the other hand, we will make use of this invariance to simplify the setting in which 
we are looking for equilibria and make accessible the methods of Crandall and Rabinowitz for operators
with simple eigenvalues. The specific simplification that we apply to our setting has also 
been employed by Escher and Matioc \cite{EM11} and is supported by 
the following proposition which allows us to consider the class
\[
F_{1, {\sf even}} := h^{4 + \alpha}_{0,{\sf even}}(\Tone)
\]
of functions which are even, i.e. symmetric about $[x = 0]$, and $h^{4 + \alpha}_0$ regular.

\begin{prop}\label{Prop:Even}
For every equilibrium $\brho$ of \eqref{PASD}, there exists $x_0 = x_0(\brho) \in \Tone$ for which
the translation $T_{x_0} \brho$ is in the space $F_{1,{\sf even}} := h^{4 + \alpha}_{0,{\sf even}}(\Tone)$
of even functions on $\Tone$ in the class $F_1$. I.e. up to translations on $\Tone$,
all equilibria of \eqref{PASD} are even.
\end{prop}

\begin{proof}
From Remarks~\ref{Rem:Equilibria} and Proposition~\ref{Prop:Equilibria}, 
we know that $\brho$ must correspond to the projection of an 
undulary curve $R(\mathcal{H},B)$, modulo translations along the $x$--axis. Choose 
$x_0 \in \Tone$ so that $T_{x_0} \brho = P_0 R(\cdot \, ; \mathcal{H},B)$ and 
one readily verifies that $R(\cdot \, ; \mathcal{H},B)$ is symmetric about $s = \pi / 2 \mathcal{H}$.
The claim follows from $x(\pi / 2 \mathcal{H}) = 0$.
\end{proof}

From this observation, we see that there is no loss of generality if we focus our bifurcation analysis
on the setting of $\trho \in F_{1,{\sf even}}$. One benefit of 
working in this setting is that we have the Fourier series representation
\[
\trho(x) = \sum_{k \geq 1} a_k \cos(kx), \quad \{ a_k \} \subset \mathbb{R} \qquad \text{for all} 
\quad \trho \in F_{1,{\sf even}} \,.
\]
We are now prepared to prove our first bifurcation result.

\begin{thm}[Bifurcation of Reduced Problem]\label{Thm:Bifurcation}
For every $\ell \in \mathbb{N}$, $(0, \ell) \in h^{4 + \alpha}_{0,{\sf even}}(\Tone) \times (0, \infty)$
is a bifurcation point for the equation \eqref{Eqn:Bifurcation}. In particular, there exists a positive
constant $\delta_{\ell} > 0$ and a nontrivial analytic curve
\begin{equation}\label{Eqn:BifurcatingBranch}
\left\{ (\trho_{\ell}(s), \lambda_{\ell}(s)) \in h^{4 + \alpha}_{0,{\sf even}} \times \mathbb{R} : s \in (- \delta_{\ell}, \delta_{\ell}), (\trho_{\ell}(0), \lambda_{\ell}(0)) = (0, \ell) \right\},
\end{equation}
such that 
\[
\bG(\trho_{\ell}(s), \lambda_{\ell}(s)) = 0 \qquad \text{for all} \quad s \in (-\delta_{\ell}, \delta_{\ell}),
\]
and all solutions of \eqref{Eqn:Bifurcation} in a neighborhood of $(0, \ell)$ are either a
trivial solution $(0, \lambda)$ or an element of the nontrivial curve \eqref{Eqn:BifurcatingBranch}.
Moreover, if $\lambda \in (0, \infty) \setminus \mathbb{N}$, then $(0, \lambda)$ is not a 
bifurcation point for \eqref{Eqn:Bifurcation}.
\end{thm}

\begin{proof}
We first note that bifurcation can only occur at points $(0, \lambda)$ for which $D_1 \bG(0, \lambda)$ is
not bijective. We can see from \eqref{Eqn:D1PG}--\eqref{Eqn:DG} that
\begin{equation}\label{Eqn:D1bG}
D_1 \bG(0, \lambda) = - \partial_x^2 \left( \lambda^2 + \partial_x^2 \right) \Big|_{F_{1,{\sf even}}},
\end{equation}
which is realized as a Fourier multiplier with the symbol
\[
\big( M_k \big)_{k \in \mathbb{N}} = \big( k^2 ( \lambda^2 - k^2 ) \big)_{k \in \mathbb{N}} \, ,
\]
and we see that the operator is bijective whenever $\lambda \in (0, \infty) \setminus \mathbb{N}$.
Hence, it follows that bifurcation can only occur at points of the form $(0, \ell)$, $\ell \in \mathbb{N}$. 

Now fix $\ell \in \mathbb{N}$ and we proceed to verify that $(0, \ell)$ is indeed 
a bifurcation point for \eqref{Eqn:Bifurcation}.
By compactness of the resolvent $R(\lambda) := (\lambda - DG_{\star}(0))^{-1}$, 
$\lambda \in \rho(DG_{\star}(0))$, it follows that 
$D_1 \bG(0, \ell)$ is a Fredholm operator of index zero.
Further, we see that
\[
\begin{split}
N_{\ell} &:= N(D_1 \bG(0, \ell)) = \text{span}\{ \cos(\ell x) \},\\ 
R_{\ell} &:= R(D_1 \bG(0, \ell)) = \overline{\text{span}} \left\{ \cos(k x): k \ge 1, k \not= \ell \right\},
\end{split}
\]
where $N(B)$ and $R(B)$ denote the kernel and the range, respectively, of the operator $B$.
Since $h^{\sigma}(\Tone) \hookrightarrow L_2(\Tone)$, we can borrow the $L_2$-inner product
to realize $N_{\ell}$ as a topological complement to $R_{\ell}$ as subspaces of $F_{1,{\sf even}}$.
Meanwhile, following from \eqref{Eqn:D1bG}, we compute the mixed derivative 
\begin{equation}\label{Eqn:D21bG}
D_2 D_1 \bG(0, \ell) = - 2  \ell  \; \partial_x^2 \Big|_{F_{1,{\sf even}}}.
\end{equation}
Now take $\hat{v}_0 := \cos(\ell \; \cdot) \in N_{\ell}$ and observe that
\[
D_2 D_1 \bG(0, \ell) \hat{v}_0 = 2 \ell^3 \, \cos(\ell \, \cdot) \notin R_{\ell} \,,
\]
from which the result follows by \cite[Theorem 1.7]{CR71}, 
or \cite[Theorem I.5.1]{KIE12}.
\end{proof}

\begin{remark}
Following from the previous result, we are able to track the behavior of the so--called 
\emph{critical eigenvalue} $\mu_{\ell}(\lambda)$ of the linearization $D_1 \bG(0, \lambda)$ 
about the trivial equilibria $(0, \lambda)$. In particular, we choose $\mu_{\ell}(\lambda)$
to be the eigenvalue of $D_1 \bG(0,\lambda)$ which passes through 
0 with non--vanishing speed at $\lambda = \ell$, the existence of $\mu_{\ell}(\lambda)$ is 
guaranteed by the bifurcation observed above, c.f. \cite[Section I.6~and~I.7]{KIE12}. 
Moreover, employing eigenvalue perturbation techniques, we can also track the 
associated perturbed eigenvalue $\hat{\mu}_{\ell}(s)$ of the linearization 
$D_1 \bG(\trho_{\ell}(s), \lambda_{\ell}(s))$ about the nontrivial equilibria. These 
eigenvalues will play a crucial role in the following instability results for the branches
of bifurcating equilibria.
\end{remark}

\begin{thm}\label{Thm:BifurcType}
Each of the bifurcations established in Theorem~\ref{Thm:Bifurcation} is a
subcritical pitchfork type bifurcation. More precisely, for all $\ell \in \mathbb{N}$,
we have 
\[
\dot{\lambda}_{\ell}(0) = 0 \quad \text{and} \quad \ddot{\lambda}_{\ell}(0) < 0,
\]
where ``$\; \, \dot{ } \;$'' denotes the derivative with respect to the parameter $s$. Moreover, 
it holds that the perturbed eigenvalues $\hat{\mu}_{\ell}(s)$ are strictly positive for
$|s| > 0$ chosen sufficiently small.
\end{thm}

\begin{proof}
Utilizing the methods of \cite[Section I.6 and I.7]{KIE12}, and the techniques
developed in the previous sections of the paper, one can explicitly verify that the bifurcations 
observed above are indeed subcritical pitchfork bifurcations. 
The result for the perturbed eigenvalues now follows from the eigenvalue perturbation
techniques in \cite[Section I.7]{KIE12}, see also Amann \cite[Section 27]{AM90}.
\end{proof}

With these bifurcation results established in the setting of the reduced problem, we
will now go about deriving results for the original problem \eqref{ASD}.
Recalling the definition of the operator $G_{\star}$ from Section~\ref{Section:Definitions}, we introduce the notation
\[
G(\rho, \lambda) := G(\rho + 1 / \lambda) = G_{\star}(\rho), \qquad \text{for} \quad \lambda = 1 / \rs.
\]
We are now interested in finding solutions to the bifurcation equation
\begin{equation}\label{Eqn:FullBifurc}
G(\rho, \lambda) = 0, \qquad (\rho, \lambda) \in h^{4 + \alpha}(\Tone) \times (0,\infty),
\end{equation}
associated with the full problem \eqref{ASD}.

We begin analyzing \eqref{Eqn:FullBifurc} by \emph{lifting} the
bifurcation results already established for the reduced problem.
We make use of the connections established in Section~\ref{Section:ReducedProblem}
and we also establish the following connection between the 
eigenvalues of $D_1 \bG$ and $D G(\cdot, \lambda)$ at equilibria.

\begin{prop}\label{Prop:Eigen}
Suppose $\bG(\trho, \lambda) = 0$ and $\mu \ne 0$. Then 
\[
\mu \text{ is an eigenvalue for } D_1 \bG(\trho, \lambda) \quad \Longleftrightarrow \quad \mu \text{ is an eigenvalue for } D_1 G(\psi(\trho), \lambda).
\]
\end{prop}

\begin{proof}
{\bf (i)} First, suppose that $D_1 \bG(\trho, \lambda) \tilde{h} = \mu \tilde{h}$ for some 
$\tilde{h} \in F_1 \setminus \{ 0 \},$ and let $h := D \psi(\trho) \tilde{h}$. Then 
$h \in E_1 \setminus \{ 0 \}$, by injectivity of $D \psi(\trho)$, and it follows from
\eqref{Eqn:GLinear2} that 
\[
D_1 G(\psi(\trho), \lambda) h = \mu h.
\]
We also observe that this assertion is true in case $\mu = 0$.

{\bf (ii)} Now suppose that $D_1 G(\psi(\trho), \lambda) h = \mu h$ for some $h \in E_1 \setminus \{ 0 \}$.
We conclude from \eqref{Eqn:GLinear1} that $h \in T_{\psi(\trho)} \mathcal{M}_0$, so that
there exists a unique $\tilde{h} \in F_1 \setminus \{ 0 \}$ for which $h = D \psi(\trho) \tilde{h}$.
Then \eqref{Eqn:GLinear2} shows that
\[
\mu D\psi(\trho) \tilde{h} = D\psi(\trho) D_1 \bG(\trho, \lambda) \tilde{h},
\]
and finally, by injectivity of $D\psi(\trho)$, we conclude that $\mu \tilde{h} = \bG(\trho, \lambda) \tilde{h}$, 
as desired.
\end{proof}

We are now prepared to prove the main result regarding bifurcation of the original 
problem \eqref{ASD} in the setting of $h^{4 + \alpha}(\Tone)$, and instability of 
the bifurcating unduloids.

\begin{thm}[Bifurcation of Full Problem]\label{Thm:FullBifurcation}
Fix $\ell \in \mathbb{N}$. Then:
\begin{enumerate}

\item the set
\begin{equation}\label{Eqn:BifurcCurve}
\Big\{ \psi(\trho_{\ell}(s)) + 1/\lambda_{\ell}(s) : s \in (-\delta_{\ell}, \delta_{\ell}) \Big\} \subset h^{4 + \alpha}(\Tone),
\end{equation}
is an analytic curve of equilibria for the problem \eqref{ASD} which bifurcates subcritically
from the family of cylinders $\rs \in (0,\infty)$, at the cylinder $\rs = 1 / \ell$.

\item there exists some $\varepsilon_{\ell} > 0$ so that for every 
$s \in (-\delta_{\ell}, \delta_{\ell})$
\[
\psi(\trho_{\ell}(s)) + 1/\lambda_{\ell}(s) = R(B, \ell), \qquad \text{for some} \quad B \in (-\varepsilon_{\ell}, \varepsilon_{\ell}),
\]
i.e. the family \eqref{Eqn:BifurcCurve} of equilibria are exactly the even presentations of 
$2 \pi / \ell$--periodic undulary curves in some neighborhood of the cylinder $\rs = 1 / \ell$. 

\item the undulary curves $R(B,\ell)$ are unstable for $|B| > 0$ chosen sufficiently small.

\end{enumerate}
\end{thm}

\begin{proof}
{\bf (a)} It follows from Proposition~\ref{Prop:Equilibria} and Theorem~\ref{Thm:Bifurcation} 
that the family
\[
\Big\{ (\psi(\trho_{\ell}(s)), \lambda_{\ell}(s)): s \in (-\delta_{\ell}, \delta_{\ell}) \Big\} \subset E_1 \times (0, \infty)
\]
consists of solutions to the bifurcation equation \eqref{Eqn:FullBifurc}.  
The regularity of the curve follows from the regularity of the bifurcating branch in Theorem~\ref{Thm:Bifurcation} 
and regularity of the mapping $\psi$. By definition of the bifurcation function $G(\cdot, \lambda)$,
it follows that the family \eqref{Eqn:BifurcCurve} are indeed equilibria of the original
equation \eqref{ASD} which intersect the family of cylinders at $\rs = 1 / \ell$, when $s = 0$. 
Meanwhile, the bifurcation parameter $\lambda$ remains unchanged in lifting from 
the reduced problem to the full problem, hence we see that
\[
\dot{\lambda}_{\ell}(0) = 0 \quad \text{and} \quad \ddot{\lambda}_{\ell}(0) < 0,
\]
from Theorem~\ref{Thm:BifurcType}, and so we conclude that the given curve bifurcates subcritically.

\medskip

{\bf (b)} By Remarks~\ref{Rem:PsiProperties}(f) it follows that $\psi$ preserves the 
symmetry of even functions on $\Tone$, and since $\trho_{\ell}(s) \in F_{1,{\sf even}}$, it follows that
the functions in the family \eqref{Eqn:BifurcCurve} are even on $\Tone$. Meanwhile, by 
the characterization of equilibria established in Section~\ref{Sec:Equilibria}, and the fact that
\[
\psi(\trho_{\ell}(0)) + 1 / \lambda_{\ell}(0) = 1 / \ell = R(0, \ell),
\]
it follows that the family \eqref{Eqn:BifurcCurve} must coincide with the family of $2 \pi/ \ell$--periodic
undulary curves $R(B,\ell)$, for some continuum of values $B \in (-\varepsilon_{\ell}, \varepsilon_{\ell})$.

\medskip

{\bf (c)} To prove that the unduloids \eqref{Eqn:BifurcCurve} are unstable, we 
mimic the proof of Theorem~\ref{Thm:Instability} in the current setting. In particular, define
\begin{align*}
&G_{\ell}(\rho, s) := G(\rho + \psi(\trho_{\ell}(s)), \lambda_{\ell}(s)), \qquad \text{and}\\
&L_{\ell}(s) := D_1 G_{\ell}(0,s) = D_1 G(\psi(\trho_{\ell}(s)), \lambda_{\ell}(s)),
\end{align*}
acting on functions $\rho \in E_1$. It follows by Theorem~\ref{Thm:BifurcType} and 
Proposition~\ref{Prop:Eigen} that 
\[
\sigma(L_{\ell}(s)) \cap [{\rm Re} \, z > 0] \ne \emptyset,
\]
provided $|s| > 0$ is chosen sufficiently small. Meanwhile, the operator 
$G_{\ell}(\cdot,s)$ has a similar quasilinear structure as $\mathcal{G}_{\star}$ 
and so the analogue to inequality \eqref{estimate-g} is also derived for
\[
g_{\ell}(\rho,s) := G_{\ell}(\rho,s) - L_{\ell}(s) \rho.
\]
Utilizing  \cite[Proposition I.7.2]{KIE12} and the explicit characterization 
\eqref{Eqn:SpectralSet} of the spectra $\sigma(DG_{\star}(\eta))$, we can control 
the eigenvalues of the perturbed linearization $L_{\ell}(s) \, ,$ so that, for sufficiently 
small values of $|s| > 0$, we can derive the necessary \emph{spectral gap} condition
\begin{equation*}
[\omega - \gamma \le {\rm Re} \,z \le \omega + \gamma] \cap \sigma(L_{\ell}(s)) = \emptyset
\quad \text{and} \quad
\sigma_+ := [{\rm Re} \, z > \omega + \gamma] \cap \sigma(L_{\ell}(s)) \ne \emptyset \, ,
\end{equation*}
for some $\gamma, \omega > 0$.
The remainder of the proof now follows as in the proof of Theorem~\ref{Thm:Instability} with
the observation that $-L_{\ell}(s)$ satisfies maximal regularity properties, which follows 
by uniform ellipticity of $L_{\ell}(s)$ and an argument similar to the proof of the stated Claim in
the proof of Lemma~\ref{Lem:AandFReg}, in the appendix. 
\end{proof}

\begin{remark}
\label{remark-bifurcation}
Note that we only prove nonlinear instability for unduloids with sufficiently small
parameter values $|B| > 0$. 
Relying on previous results in
\cite{ATH87, VOG87, VOG89} for the stationary {\em trapped drop} capillary problem,
the authors of \cite{BBW98} observe that the linearized problem
at nontrivial unduloids always has an unstable eigenvalue.
Applying this observation in our setting, we can in fact conclude nonlinear
instability of the entire family of nontrivial unduloids $R(B,k), |B| \in (0,1), k \in \mathbb{N}$.
\end{remark}

\section{Appendix}
In this section we outline the proofs of
Lemma~\ref{Lem:AandFReg} and 
Propositions~\ref{Prop:Existence} and \ref{Prop:SolutionRegularity},
see \cite{LeCD} for more details. 
\begin{proof}[Proof of Lemma \ref{Lem:AandFReg}] 
Fix $\mu \in [1/2,1]$ as indicated.
\medskip\\
{\bf CLAIM:} $\mathcal{A}(\rho) \in \mathcal{M}R_{\nu}(E_1,E_0)$ for $\rho \in V_{\mu}\, , \; \nu \in (0, 1].$
This claim will follow from \cite{LeC11}, though the setting of that paper differs slightly from the current
setting and warrants a brief discussion. First, for $\rho \in V_{\mu}$ define the coefficients
\[
b_4(\rho) := \frac{1}{(1 + \rho_x^2)^2} \quad \text{and} \quad 
b_3(\rho) := \frac{2 \rho_x \big( 1 + \rho_x^2 - 3\rho \rho_{xx} \big)}{\rho \big( 1 + \rho_x^2 \big)^3},
\]
so that $\mathcal{A}(\rho) = b_4(\rho) \, \partial_x^4 + b_3(\rho) \, \partial_x^3$. By our
choice of $\mu$, it follows that $V_{\mu} \subset h^{2 + \alpha}(\Tone, \mathbb{R})$, so that $b_4, b_3 \in E_0$ and 
$\mathcal{A}(\rho)$ is a uniformly elliptic differential operator. By \cite[Theorem 5.2]{LeC11}  
we conclude that 
\[
\mathcal{A}(\rho) \in \mathcal{MR}_{\nu} \big( h^{4 + \alpha}(\Tone, \mathbb{C}), h^{\alpha}(\Tone, \mathbb{C}) \big),
\qquad \nu \in (0, 1],
\]
where we utilize the notation $h^{k + \alpha}(\Tone, \mathbb{C})$ to be clear that the space consists 
of $\mathbb{C}$--valued functions over $\Tone$, and does not coincide with the spaces $E_{\mu}$ 
being considered herein. However, $h^{k + \alpha}(\Tone, \mathbb{C})$ does
coincide with the complexification of $h^{k + \alpha}(\Tone, \mathbb{R})$ (up to equivalent norms) and it is a 
straightforward exercise to see that the property of maximal regularity continues to hold under restriction
to the subspaces $h^{\sigma}(\Tone, \mathbb{R})$.

\medskip

The regularity assertion for $(\mathcal{A},f)$ follows from the 
fact that the mappings
\begin{equation*}
\left[ r \mapsto 1/r \right]: V_0 \to  E_0, \;
[r \mapsto r_x]: h^{\sigma + 1}(\Tone) \to h^{\sigma}(\Tone), \;
[(r,s) \mapsto rs]: E_0 \times E_0\to E_0
\end{equation*}
are real analytic, and the additional observation that the mapping 
$\mathcal{A}: V_{\mu} \rightarrow \mathcal{L}(E_1,E_0)$ inherits the regularity
of the coefficients $b_3, b_4: V_{\mu} \rightarrow E_0$ and the fact that
$\mathcal{M}R_{\mu}(E_1, E_0)$ is an open subset of $\mathcal{L}(E_0, E_1)$, 
c.f. \cite[Lemma 2.5(a)]{CS01}.
\end{proof}

\begin{proof}[Proof of Proposition \ref{Prop:Existence}]
In case $\mu \in [1/2, 1)$, the result follows from Lemma~\ref{Lem:AandFReg} and 
\cite[Theorems 4.1, 5.1 and 6.1]{CS01}. When $\mu = 1$ we note that the existence
and uniqueness of a maximal solution 
\[
r(\cdot, r_0) \in C^1(J(r_0), E_0) \cap C(J(r_0), E_1)
\]
follows from \cite[Theorem 4.1(b)]{CS01}. However, for the semiflow properties, we
will consider \eqref{ASD} as a 
fully nonlinear equation, and apply results of Angenent \cite{AN90}. In particular, 
for $r \in V_1$ we use the representation $G(r) = -\mathcal{A}(r)r + f(r)$ and 
\eqref{Eqn:ADefined}--\eqref{Eqn:FDefined} to see that the Fr\'echet derivative $DG$
has the structure
\[
DG(r) = - \frac{1}{(1 + r_x^2)^2} \; \partial_x^4 + \sum_{k = 0}^3 B_k(r) \; \partial_x^k,
\]
where the coefficients $B_k(r) \in E_0$ for every $r \in V_1$, $k = 0, \ldots, 3$.
From this computation it follows that $-DG(r)$ is a uniformly elliptic operator from 
$E_1$ to $E_0$ and so, using the results of \cite{LeC11} as in the highlighted Claim 
in the proof of
Lemma~\ref{Lem:AandFReg} above, we see that $-DG(r) \in \mathcal{M}R_1(E_1, E_0)$
for all $r \in V_1$. Now the fact that \eqref{ASD} generates an analytic semiflow
on $V_1$ follows from \cite[Corollary 2.9]{AN90}.
\end{proof}

\begin{proof}[Proof of Proposition \ref{Prop:SolutionRegularity}]
For $a \in \mathbb{R}$ let $T_a : \Tone \to \Tone$ be the translation operator, where
$T_a(x)$ denotes the unique element in $\Tone$ that is in the coset 
$[x + a] \in \mathbb{R}/2 \pi \mathbb{Z}$ of $(x + a)$. 
$T_a$ naturally acts on functions $u \in C(\Tone, \mathbb{R})$ by virtue of 
$(T_a u)(x) := u(T_a (x))$.
As in \cite{ES96} one shows that, for $a \in \mathbb{R}$, the family of translations
$\{T_{ta} : t \in \mathbb{R} \}$ induces a strongly continuous group of contractions
on any of the spaces $E_{\mu}$, with infinitesimal generator $A_a$
given by
\[
 D(A_a) = h^{1 + 4 \mu + \alpha}(\Tone, \mathbb{R}), \qquad A_a = a \partial_x.
\]

Let $r_0 \in V_{\mu}$ be fixed, and let 
\[
r = r(\cdot, r_0) \in C_{1 - \mu}^1(J(r_0), E_0) \cap C_{1 - \mu}(J(r_0), E_1)
\]
be the unique solution to \eqref{ASD2}
on the maximal interval of existence $J(r_0) = [0, t^+(r_0))$.
Let $t_1 \in (0,t^+(r_0))$ be fixed and set $I := [0, t_1]$. 
Then there exists  $\delta > 0$ such that $(1 + \lambda) t \in J(r_0)$ 
for all $(t, \lambda) \in I \times (-\delta, \delta)$.
Finally, for  $(\lambda, a) \in W := (-\delta,\delta)^2$ we set
\[
r_{\lambda, a}(t) := T_{ta} r((1 + \lambda) t), \qquad t \in I;
\]
i.e. $r_{\lambda, a}(t,x) = r((1 + \lambda) t, T_{ta}(x) )$ for
$(t,x) \in I \times \Tone$.
One verifies that 
\[
r_{\lambda, a} \in \mathbb{E}_1(I) := BU\!C_{1 - \mu}^1(I, E_0) \cap BU\!C_{1 - \mu}(I, E_1).
\]
Moreover, since the nonlinear mapping $[r \mapsto G(r)]$ is equivariant with respect to translations,
i.e. $T_{b} \, G(r) = G(T_b \, r)$ for any $b \in \mathbb{R}$, we obtain that 
$r_{ \lambda, a}$ is a solution of the parameter-dependent equation
\begin{equation}
\label{parameter-dependent}
\begin{cases} \partial_t v = (1 + \lambda) G(v) + a \partial_x v, &\text{$t > 0$,}\\ v(0) = r_0, \end{cases}
\end{equation}
on the time interval $I$.

Now, for $\mathbb{U}(I) := \mathbb{E}_1(I) \cap C(I,V)$ 
we define
\begin{equation*}
\Phi : \mathbb{U}(I) \times W \to \mathbb{E}_0(I) \times E_\mu,
\quad \Phi(v, (\lambda, a)) = \big( \partial_t v - (1 + \lambda)G(v) - a \partial_x v,\, \gamma v - r_0 \big), 
\end{equation*}
where $\mathbb{E}_0(I) := BU\!C_{1 - \mu}(I, E_0)$, 
and we note that $\Phi(r_{\lambda, a}, (\lambda, a)) = (0,0)$.
Moreover,
\begin{equation*}
\Phi \in C^{\omega} \Big( \mathbb{U}(I) \times W, \mathbb{E}_0(I) \times E_\mu \Big), \quad
D_1 \Phi(r, (0,0))= \left( \frac{d}{dt}- DG(r), \gamma \right),
\end{equation*}
where we use the same notation for $r = r(\cdot, r_0)$ and its restriction to the time interval $I$.
Exactly as in the proof of \cite[Theorem 6.1]{CS01} one shows that
\begin{equation*}
D_1 \Phi(r, (0,0)) \in \mathcal{L}_{isom}(\mathbb{E}_1(I), \mathbb{E}_0(I) \times E_{\mu}).
\end{equation*}
Finally, according to the implicit function theorem, 
c.f. \cite[Theorem 15.3]{DE85} or \cite[(10.2.1)]{DI60}, there exist 
a neighborhood of $r$ in $\mathbb{E}_1(I)$ and a neighborhood of $(0,0)$ in $\mathbb{R}^2$,
which we will again denote by $\mathbb{U}(I)$ and $W$, respectively, 
and a mapping
$g \in C^{\omega}(W,\mathbb{E}_1(I))$
such that
\begin{equation*}
\Phi(v,(\lambda,a))=(0,0)\quad\text{if and only if}
\quad v=g(\lambda,a)
\end{equation*} 
whenever $(v,(\lambda,a))\in\mathbb{U}(I)\times W$.
We conclude that $g(\lambda, a) = r_{\lambda, a}$ and 
\begin{equation}
[(\lambda,a) \mapsto r_{\lambda,a}]\in C^\omega(W,\mathbb{U}(I)).
\end{equation}
For $t_0 \in (0, t_1)$ and $x_0 \in \Tone$ fixed,
we see that
\begin{equation}
[(\lambda, a) \mapsto r((1 + \lambda) t_0, T_{t_0 a}(x_0) )]
\in C^{\omega}(W, \mathbb{R}),
\end{equation}
and the assertion follows
since $(t_0, x_0)$ can be chosen arbitrarily.
\end{proof}

\section*{Acknowledgements}
The first author would like to thank the Department of Mathematics at Vanderbilt University
for its support during the preparation of this manuscript, a slightly altered version of which 
appears as a chapter in my PhD Dissertation. The authors would also like to thank Juraj
Foldes for helpful conversations regarding the axisymmetric surface diffusion flow.
The authors are thankful to the anonymous reviewers for suggestions that
helped to improve the presentation of the paper.


\end{document}